\documentclass{article}
\usepackage{amsmath}
\usepackage[ansinew]{inputenc}
\usepackage{amssymb}
\usepackage{amsthm}
\usepackage{amsfonts}
\usepackage{amscd}
\usepackage{latexsym, amsfonts}
\usepackage{graphicx}
\usepackage{tikz}

\title{Mixed labyrinth fractals}

\author{Ligia L. Cristea \thanks{This author is supported by the Austrian Science Fund (FWF), 
Project P27050-N26 and by the Austrain-French cooperation project FWF I1136-N26.} 
\\Karl-Franzens-Universit\"at Graz\\ Institut f\"ur Mathematik und Wissenschaftliches Rechnen
\\Heinrichstrasse 36, 8010 Graz\\Austria\\ \tt{strublistea@gmail.com} 
\and Bertran Steinsky\\ F\"urbergstr. 56, 5020 Salzburg\\Austria\\ \tt{steinsky@finanz.math.tu-graz.ac.at} }

\newtheorem{theorem}{Theorem}
\newtheorem{lemma}{Lemma}
\newtheorem{property}{Property}
\newtheorem{proposition}{Proposition}
\newtheorem{conjecture}{Conjecture}
\newtheorem{corollary}{Corollary}

\newcommand{\A}{\includegraphics{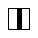}}
\newcommand{\B}{\includegraphics{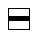}}
\newcommand{\C}{\includegraphics{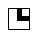}}
\newcommand{\D}{\includegraphics{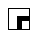}}
\newcommand{\E}{\includegraphics{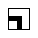}}
\newcommand{\F}{\includegraphics{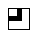}}
\begin{document}

\maketitle

\abstract{Labyrinth fractals are self-similar fractals that were introduced and studied in
recent work \cite{laby_4x4, laby_oigemoan}. In the present paper we define and study more general objects, 
called \emph{mixed labyrinth fractals}, that are in general not self-similar and are constructed by using 
sequences of \emph{labyrinth patterns}. We show that mixed labyrinth fractals are dendrites and study 
properties of the paths in the graphs associated to prefractals, and of arcs in the fractal, e.g., 
 the path length and the box counting dimension and length of arcs. We also consider more general objects related to mixed labyrinth 
 fractals, formulate two conjectures about arc lengths, and establish connections to recent results on generalised 
Sierpi\'nski carpets.}
\section{Introduction}\label{sec:introduction}
Labyrinth fractals are self-similar fractals that were introduced and studied by 
Cristea and Steinsky \cite{laby_4x4, laby_oigemoan}.  
In the present paper we deal with \emph{mixed labyrinth fractals} that are a generalisation of the 
labyrinth fractals studied before \cite{laby_4x4, laby_oigemoan}. Mixed labyrinth fractals are fractal 
sets obtained by an iterative 
construction that uses labyrinth patterns, as described in Sections~\ref{sec:Construction} and \ref{sec:Definition}, and are in general not self-similar.
We remind that generalised Sierpi\'nski carpets \cite{connected_general_carpets,totally_disco} studied 
some years ago were also defined with the help of patterns, and are in general not self-similar. 
There are recent results \cite{Luo} on the topology of a class of self-similar Sierpi\'nski carpets called \emph{fractal squares}.
In the case of the mixed labyrinth fractals, there are special restrictions on the patterns, 
that correspond to the properties of labyrinth sets \cite{laby_4x4, laby_oigemoan}. 
Labyrinth patterns have three properties, which we formulate in Section~\ref{sec:Definition}, with the help 
of graphs that we associate to the patterns. 
An example for the first two steps of the construction of a mixed labyrinth fractal is 
illustrated in Figure \ref{fig:A1A2} and \ref{fig:W2}. In Section \ref{sec:Topological properties of mixed labyrinth fractals} we show that mixed 
labyrinth fractals are dendrites. 
Section \ref{sec:Paths} is dedicated to properties of the paths in the graphs associated to the prefractals and of paths in 
the fractal.
In Section \ref{sec:Blocked} we conjecture a property of the length of the path between any two points in a mixed labyrinth fractal 
generated by a sequence of \emph{blocked} labyrinth patterns. In the Sections \ref{sec:Rectangular} and 
\ref{sec:Wild} we discuss other extensions of labyrinth sets and labyrinth fractals. Finally, 
Section \ref{sec:Totally disco} is dedicated to a result on the total disconnectedness of 
generalised Sierpi\'nski carpets generated by complementary patterns of labyrinth patterns. 

Before presenting our results let us remark that labyrinth fractals are strongly related to other mathematical objects studied by mathematicians and physicists. First, we mention that during the last years related objects called ``fractal labyrinths'' have been in the attention of physicists, either in the context of nanostructures \cite{GrachevPotapovGerman2013}, or of fractal reconstructions of images, signals and radar backgrounds \cite{PotapovGermanGrachev2013}, 
who used the formalism for the self-similar case of labyrinth fractals \cite{laby_4x4}  in order to refine their  definition of these objects, or to formulate it with more precision. We note that while our main interest regards properties of the limit set, the physicists focus on what we would call the prefractals of a labyrinth fractal, i.e., on the objects obtained  after a finite numer of steps in the iterative construction of the fractal.

Mixed labyrinth fractals are also related  to the objects introduced 1986 by Falconer \cite{Falconer_randomfractals1986} in the context of random fractals, as \emph{net fractals},  and \emph{random net fractals}. In that probabilistic framework the focus is on the Hausdorff dimension of these random fractals. There, trees and random networks are used in order to define these objects and study their dimension. Due to the very general setting chosen for the mixed labyrinth fractals in the present paper, e.g., to the fact that there are no restrictions on the width of the patterns or on the number of black/white squares in the patterns that generate the mixed labyrinth fractal, we could not apply the results on measures and dimensions obtained for net fractals to the mixed labyrinth fractals. In the same context we also mention the work of Mauldin and Williams on random fractals \cite{MauldinWilliams_random_1986} and on graph directed constructions \cite{MauldinWilliams_graphdirected_1988}.
Moreover, the graph directed Markov systems (GDMS) studied in much detail in the book of Mauldin and Urbanski \cite{MauldinUrbanski_bookGDMS_2003} are also related to the objects that we study here. 
However, in the very general setting of the present paper, regarding the width and the structure of the patterns that generate the mixed labyrinth fractal, we chose to extend the results obtained  for self-similar labyrinth fractals to the case of mixed labyrinth fractals by reasoning whithin the same framework.
For an overview on random fractals we also refer to M\"orters' contributed chapter to  the volume on new perspectives in stochastic geometry \cite{book_stochasticgeometry_2010}.

Finally, let us also mention that there is a lot of ongoing research on
$V$-variable fractals, e.g., \cite{freiberghamblyhutchinson_Vvariable}, that also provide a framework that could be used for certain classes of generalised Sierpi\'nski carpets, and, in particular, families of mixed labyrinth fractals. For $V$-variable fractals and superfractals we also refer to Barnsley's book \cite{barnsley_superfractals}.

To conclude this introductive section, let us remark that in the present paper we focus on geometric and topological properties of mixed labyrinth fractals as a generalisation of the self-similar labyrinth fractals studied before \cite{laby_4x4, laby_oigemoan} and stick to the same setting, to the approach with patterns and no probabilistic frame. Of course, in future work one can use the probabilistic, GDMS, or $V$-variable fractals approach, in order to achieve further results on classes of mixed labyrinth fractals.
\section{Construction}\label{sec:Construction}

In order to construct labyrinth fractals we use \emph{patterns}.
Figures~\ref{fig:A1A2} and \ref{fig:W2} show examples of patterns and illustrate the first two steps of the 
construction described now. 
Let $x,y,q\in [0,1]$ such that $Q=[x,x+q]\times [y,y+q]\subseteq [0,1]\times [0,1]$. 
Then for any point $(z_x,z_y)\in[0,1]\times [0,1]$ we define the function
\[
P_Q(z_x,z_y)=(q z_x+x,q z_y+y).
\]
Let $m\ge 1$. $S_{i,j,m}=\{(x,y)\mid \frac{i}{m}\le x \le \frac{i+1}{m} \mbox{ and } \frac{j}{m}\le y \le \frac{j+1}{m} \}$ and  
${\cal S}_m=\{S_{i,j,m}\mid 0\le i\le m-1 \mbox{ and } 0\le j\le m-1 \}$. 

We call any nonempty ${\cal A} \subseteq {\cal S}_m$ an $m$-\emph{pattern} and $m$ its \emph{width}. Let $\{{\cal A}_k\}_{k=1}^{\infty}$ 
be a sequence of non-empty patterns and $\{m_k\}_{k=1}^{\infty}$ be the corresponding 
\emph{width-sequence}, i.e., for all $k\ge 1$ we have 
${\cal A}_k\subseteq {\cal S}_{m_k}$. 
We denote $m(n)=\prod_{k=1}^n m_k$, for all $n \ge 1$. 
We let ${\cal W}_1={\cal A}_{1}$, and call it the 
\emph{set of white squares of level $1$}. 
Then we define ${\cal B}_1={\cal S}_{m_1} \setminus {\cal W}_1$ 
as the \emph{set of black squares of level $1$}. 
For $n\ge 2$ we define the \emph{set of white squares of level $n$} by 
\begin{equation} \label{eq:W_n}
{\cal W}_n=\bigcup_{W\in {\cal A}_{n}, W_{n-1}\in {\cal W}_{n-1}}\{ P_{W_{n-1}}(W)\}.
\end{equation}

We note that ${\cal W}_n\subset {\cal S}_{m(n)}$, and we define the \emph{set of black squares of level $n$} by ${\cal B}_n={\cal S}_{m(n)} \setminus {\cal W}_n$. For $n\ge 1$, we define $L_n=\bigcup_{W\in {\cal W}_n} W$. 
Therefore, $\{L_n\}_{n=1}^{\infty}$ is a monotonically decreasing sequence of compact sets. 
We write $L_{\infty}=\bigcap_{n=1}^{\infty}L_n$, i.e., the \emph{limit set defined by the sequence of patterns 
$\{{\cal A}_k\}_{k=1}^{\infty}.$ }                                                                                                                                                                                                                                                                                                                                                                                                                                                                                                                                                                                                                                                                                                                                                                                                                                                                                                                                                                                                                              
                                                                                                                                                                                                                                                                                                                                                                                                                                                                                                                                                                                                                                                                                                                                                                                                                                                                                    
\section{Definition of mixed labyrinth fractals}\label{sec:Definition}

A {\em graph} $\mathcal{G}$ is a pair $(\mathcal{V},\mathcal{E})$, where $\mathcal{V}=\mathcal{V}(\mathcal{G})$ is a finite set of vertices, and the set of edges $\mathcal{E}=\mathcal{E}(\mathcal{G})$ is a subset of $\{\{u,v\} \mid u,v \in \mathcal{V}, u\neq v\}$. We write $u \sim v$ if $\{u,v\}\in\mathcal{E}(\mathcal{G})$ and we say $u$ is a \emph{neighbour} of $v$. The sequence of vertices $\{u_i\}_{i=0}^{n}$ is a \emph{path between $u_0$ and $u_n$} in a graph $\mathcal{G}\equiv (\mathcal{V},\mathcal{E})$, if $u_0,u_1,\ldots,u_n\in \mathcal{V}$, $u_{i-1}\sim u_i$ for $1 \le i\le n$, and $u_i\neq u_j$ for $0\le i<j\le n$.  The sequence of vertices $\{u_i\}_{i=0}^{n}$ is a \emph{cycle} in $\mathcal{G}\equiv (\mathcal{V},\mathcal{E})$, if $u_0,u_1,\ldots,u_n\in \mathcal{V}$, $u_{i-1}\sim u_i$ for $1 \le i\le n$, $u_i\neq u_j$ for $1\le i<j\le n$, and $u_0=u_n$. A \emph{tree} 
is a connected graph that contains no cycle. A \emph{connected component} is an equivalence class of the relation, where two vertices 
are related if there is a path between them.
For ${\cal A}\subseteq {{\cal S}_m}$, we define $\mathcal{G}({\cal A})\equiv (\mathcal{V}(\mathcal{G}({\cal A})),\mathcal{E}(\mathcal{G}({\cal A})))$ to be the graph of ${\cal A}$, i.e., the graph whose vertices $\mathcal{V}(\mathcal{G}({\cal A}))$ are the white squares in ${\cal A}$, and whose edges $\mathcal{E}(\mathcal{G}({\cal A}))$ are the unordered pairs of white squares, that have a common side. The \emph{top row} in ${\cal A}$ is the set of all white squares in $\{S_{i,m-1,m}\mid 0\le i\le m-1 \}$. The bottom row, left column, and right column in ${\cal A}$  are defined analogously. A \emph{top exit} in ${\cal A}$ is a white square in the top row, such that there is a white square in the same column in the bottom row. A \emph{bottom exit} in ${\cal A}$  is defined analogously. A \emph{left exit} in ${\cal A}$ is a white square in the left column, such that there is a white square in the same row in the right column. A \emph{right exit} in ${\cal A}$ is defined analogously. 
While a top exit together with the corresponding bottom exit build a \emph{vertical exit pair}, a left exit and the corresponding right exit build a \emph{horizontal exit pair}.

%
%
If ${\cal A}={\cal W}_{n}$,  for $n\ge 1$, we call the top row in ${\cal A}$ the \emph{top row of level} $n$. The \emph{bottom row, left column}, and \emph{right column of level} $n$ are defined analogously.

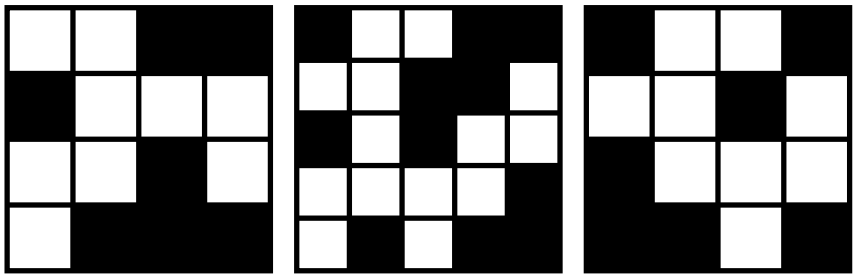
\begin{figure}[hhhh]\label{A1A2}
\begin{center}
\begin{tikzpicture}[scale=.35]
\draw[line width=2pt] (0,0) rectangle (10,10);
\draw[line width=2pt] (2.5, 0) -- (2.5,10);
\draw[line width=2pt] (5, 0) -- (5,10);
\draw[line width=2pt] (7.5, 0) -- (7.5,10);
\draw[line width=2pt] (0, 2.5) -- (10,2.5);
\draw[line width=2pt] (0, 5) -- (10,5);
\draw[line width=2pt] (0, 7.5) -- (10,7.5);
\filldraw[fill=black, draw=black] (2.5,0) rectangle (5, 2.5);
\filldraw[fill=black, draw=black] (5,0) rectangle (7.5, 2.5);
\filldraw[fill=black, draw=black] (7.5,0) rectangle (10, 2.5);
\filldraw[fill=black, draw=black] (5,2.5) rectangle (7.5, 5);
\filldraw[fill=black, draw=black] (0,5) rectangle (2.5, 7.5);
\filldraw[fill=black, draw=black] (5,7.5) rectangle (7.5, 10);
\filldraw[fill=black, draw=black] (7.5,7.5) rectangle (10, 10);
\draw[line width=2pt] (11,0) rectangle (21,10);
\draw[line width=2pt] (13, 0) -- (13,10);
\draw[line width=2pt] (15, 0) -- (15,10);
\draw[line width=2pt] (17, 0) -- (17,10);
\draw[line width=2pt] (19, 0) -- (19,10);
\draw[line width=2pt] (11, 2) -- (21,2);
\draw[line width=2pt] (11, 4) -- (21,4);
\draw[line width=2pt] (11, 6) -- (21,6);
\draw[line width=2pt] (11, 8) -- (21,8);
\filldraw[fill=black, draw=black] (13,0) rectangle (15, 2);
\filldraw[fill=black, draw=black] (17,0) rectangle (19, 2);
\filldraw[fill=black, draw=black] (19,0) rectangle (21, 2);
\filldraw[fill=black, draw=black] (19,2) rectangle (21, 4);
\filldraw[fill=black, draw=black] (11,4) rectangle (13, 6);
\filldraw[fill=black, draw=black] (15,4) rectangle (17, 6);
\filldraw[fill=black, draw=black] (15,6) rectangle (17, 8);
\filldraw[fill=black, draw=black] (17,6) rectangle (19, 8);
\filldraw[fill=black, draw=black] (11,8) rectangle (13, 10);
\filldraw[fill=black, draw=black] (17,8) rectangle (19, 10 );
\filldraw[fill=black, draw=black] (19,8) rectangle (21, 10);
\draw[line width=2pt] (22,0) rectangle (32,10);
\draw[line width=2pt] (24.5, 0) -- (24.5,10);
\draw[line width=2pt] (27,0) -- (27,10);
\draw[line width=2pt] (29.5, 0) -- (29.5,10);
\draw[line width=2pt] (22, 2.5) -- (32,2.5);
\draw[line width=2pt] (22, 5) -- (32,5);
\draw[line width=2pt] (22, 7.5) -- (32,7.5);
\filldraw[fill=black, draw=black] (22,0) rectangle (24.5, 5);
\filldraw[fill=black, draw=black] (22,7.5) rectangle (24.5, 10);
\filldraw[fill=black, draw=black] (24.5,0) rectangle (27, 2.5);
\filldraw[fill=black, draw=black] (27,5) rectangle (29.5, 7.5);
\filldraw[fill=black, draw=black] (29.5,7.5) rectangle (32, 10);
\filldraw[fill=black, draw=black] (29.5,0) rectangle (32, 2.5);
\end{tikzpicture}
\caption{Three labyrinth patterns, ${\cal A}_1$ (a $4$-pattern), ${\cal A}_2$ (a $5$-pattern), and ${\cal A}_3$ (a $4$-pattern)}\label{fig:A1A2}
\end{center}
\end{figure}

A non-empty $m$-pattern ${\cal A} \subseteq {{\cal S}_m}$, $m \ge 3$ is called a $m\times m$-\emph{labyrinth pattern} (in short, \emph{labyrinth pattern}) if  ${\cal A}$ satisfies Property~\ref{prop1}, Property~\ref{prop2}, and Property~\ref{prop3}.
\begin{property}\label{prop1}
$\mathcal{G}({\cal A})$ is a tree.
\end{property}
\begin{property}\label{prop2}
${\cal A}$ has exactly one vertical exit pair, and exactly one horizontal exit pair.
\end{property}
\begin{property}\label{prop3}
If there is a white square in ${\cal A}$ at a corner of ${\cal A}$, then there is no white square in ${\cal A}$ at the diagonally opposite corner of ${\cal A}$. 
\end{property}

Let $\{{\cal A}_k\}_{k=1}^{\infty}$ be a sequence of non-empty patterns,  with $m_k\ge 3$, $n\ge 1$ and ${\cal W}_n$ the corresponding set of white squares of level $n$. We call ${\cal W}_{n}$ an $m(n)\times m(n)$-\emph{mixed labyrinth set} (in short, \emph{labyrinth set}), if ${\cal A} ={\cal W}_{n}$ satisfies Property~\ref{prop1}, Property~\ref{prop2}, and Property~\ref{prop3}.
\begin{figure}[hhhh]\label{W2}
\begin{center}
\begin{tikzpicture}[scale=.2]
\draw[line width=1pt] (0,0) rectangle (20,20);
\draw[line width=0.8pt] (5, 0) -- (5,20);
\draw[line width=0.8pt] (10, 0) -- (10,20);
\draw[line width=0.8pt] (15, 0) -- (15,20);
\draw[line width=0.8pt] (0, 5) -- (20,5);
\draw[line width=0.8pt] (0, 10) -- (20,10);
\draw[line width=0.8pt] (0, 15) -- (20,15);
\draw[line width=0.5pt] (1, 0) -- (1,20);
\draw[line width=0.5pt] (2, 0) -- (2,20);
\draw[line width=0.5pt] (3, 0) -- (3,20);
\draw[line width=0.5pt] (4, 0) -- (4,20);
\draw[line width=0.5pt] (6, 0) -- (6,20);
\draw[line width=0.5pt] (7, 0) -- (7,20);
\draw[line width=0.5pt] (8, 0) -- (8,20);
\draw[line width=0.5pt] (9, 0) -- (9,20);
\draw[line width=0.5pt] (11, 0) -- (11,20);
\draw[line width=0.5pt] (12, 0) -- (12,20);
\draw[line width=0.5pt] (13, 0) -- (13,20);
\draw[line width=0.5pt] (14, 0) -- (14,20);
\draw[line width=0.5pt] (16, 0) -- (16,20);
\draw[line width=0.5pt] (17, 0) -- (17,20);
\draw[line width=0.5pt] (18, 0) -- (18,20);
\draw[line width=0.5pt] (19, 0) -- (19,20);
\draw[line width=0.5pt] (0, 1) -- (20,1);
\draw[line width=0.5pt] (0, 2) -- (20,2);
\draw[line width=0.5pt] (0, 3) -- (20,3);
\draw[line width=0.5pt] (0, 4) -- (20,4);
\draw[line width=0.5pt] (0, 6) -- (20,6);
\draw[line width=0.5pt] (0, 7) -- (20,7);
\draw[line width=0.5pt] (0, 8) -- (20,8);
\draw[line width=0.5pt] (0, 9) -- (20,9);
\draw[line width=0.5pt] (0, 11) -- (20,11);
\draw[line width=0.5pt] (0, 12) -- (20,12);
\draw[line width=0.5pt] (0, 13) -- (20,13);
\draw[line width=0.5pt] (0, 14) -- (20,14);
\draw[line width=0.5pt] (0, 16) -- (20,16);
\draw[line width=0.5pt] (0, 17) -- (20,17);
\draw[line width=0.5pt] (0, 18) -- (20,18);
\draw[line width=0.5pt] (0, 19) -- (20,19);
\filldraw[fill=black, draw=black] (5,0) rectangle (10, 5);
\filldraw[fill=black, draw=black] (10,0) rectangle (15, 5);
\filldraw[fill=black, draw=black] (15,0) rectangle (20, 5);
\filldraw[fill=black, draw=black] (10,5) rectangle (15, 10);
\filldraw[fill=black, draw=black] (0,10) rectangle (5, 15);
\filldraw[fill=black, draw=black] (10,15) rectangle (15, 20);
\filldraw[fill=black, draw=black] (15,15) rectangle (20, 20);
\filldraw[fill=black, draw=black] (1,0) rectangle (2, 1);
\filldraw[fill=black, draw=black] (3,0) rectangle (4, 1);
\filldraw[fill=black, draw=black] (4,0) rectangle (5, 1);
\filldraw[fill=black, draw=black] (4,1) rectangle (5, 2);
\filldraw[fill=black, draw=black] (0,2) rectangle (1, 3);
\filldraw[fill=black, draw=black] (2,2) rectangle (3, 3);
\filldraw[fill=black, draw=black] (2,3) rectangle (3, 4);
\filldraw[fill=black, draw=black] (3,3) rectangle (4, 4);
\filldraw[fill=black, draw=black] (0,4) rectangle (1, 5);
\filldraw[fill=black, draw=black] (3,4) rectangle (4, 5 );
\filldraw[fill=black, draw=black] (4,4) rectangle (5, 5);
\filldraw[fill=black, draw=black] (1,5) rectangle (2, 6);
\filldraw[fill=black, draw=black] (3,5) rectangle (4, 6);
\filldraw[fill=black, draw=black] (4,5) rectangle (5, 6);
\filldraw[fill=black, draw=black] (4,6) rectangle (5, 7);
\filldraw[fill=black, draw=black] (0,7) rectangle (1, 8);
\filldraw[fill=black, draw=black] (2,7) rectangle (3, 8);
\filldraw[fill=black, draw=black] (2,8) rectangle (3, 9);
\filldraw[fill=black, draw=black] (3,8) rectangle (4, 9);
\filldraw[fill=black, draw=black] (0,9) rectangle (1, 10);
\filldraw[fill=black, draw=black] (3,9) rectangle (4, 10);
\filldraw[fill=black, draw=black] (4,9) rectangle (5, 10);
\filldraw[fill=black, draw=black] (6,5) rectangle (7, 6);
\filldraw[fill=black, draw=black] (8,5) rectangle (9, 6);
\filldraw[fill=black, draw=black] (9,5) rectangle (10, 6);
\filldraw[fill=black, draw=black] (9,6) rectangle (10, 7);
\filldraw[fill=black, draw=black] (5,7) rectangle (6, 8);
\filldraw[fill=black, draw=black] (7,7) rectangle (8, 8);
\filldraw[fill=black, draw=black] (7,8) rectangle (8, 9);
\filldraw[fill=black, draw=black] (8,8) rectangle (9, 9);
\filldraw[fill=black, draw=black] (5,9) rectangle (6, 10);
\filldraw[fill=black, draw=black] (8,9) rectangle (9, 10);
\filldraw[fill=black, draw=black] (9,9) rectangle (10, 10);
\filldraw[fill=black, draw=black] (6,10) rectangle (7, 11);
\filldraw[fill=black, draw=black] (8,10) rectangle (9, 11);
\filldraw[fill=black, draw=black] (9,10) rectangle (10, 11);
\filldraw[fill=black, draw=black] (9,11) rectangle (10, 12);
\filldraw[fill=black, draw=black] (7,12) rectangle (6, 13);
\filldraw[fill=black, draw=black] (7,12) rectangle (8, 13);
\filldraw[fill=black, draw=black] (7,13) rectangle (8, 14);
\filldraw[fill=black, draw=black] (8,13) rectangle (9, 14);
\filldraw[fill=black, draw=black] (5,14) rectangle (6, 15);
\filldraw[fill=black, draw=black] (8,14) rectangle (9, 15 );
\filldraw[fill=black, draw=black] (9,14) rectangle (10, 15);
\filldraw[fill=black, draw=black] (11,10) rectangle (12, 11);
\filldraw[fill=black, draw=black] (13,10) rectangle (14, 11);
\filldraw[fill=black, draw=black] (14,10) rectangle (15, 11);
\filldraw[fill=black, draw=black] (14,11) rectangle (15, 12);
\filldraw[fill=black, draw=black] (10,12) rectangle (11, 13);
\filldraw[fill=black, draw=black] (12,12) rectangle (13, 13);
\filldraw[fill=black, draw=black] (12,13) rectangle (13, 14);
\filldraw[fill=black, draw=black] (13,13) rectangle (14, 14);
\filldraw[fill=black, draw=black] (10,14) rectangle (11, 15);
\filldraw[fill=black, draw=black] (13,14) rectangle (14, 15 );
\filldraw[fill=black, draw=black] (14,14) rectangle (15, 15);
\filldraw[fill=black, draw=black] (16,5) rectangle (17, 6);
\filldraw[fill=black, draw=black] (18,5) rectangle (19, 6);
\filldraw[fill=black, draw=black] (19,5) rectangle (20, 6);
\filldraw[fill=black, draw=black] (19,6) rectangle (20, 7);
\filldraw[fill=black, draw=black] (15,7) rectangle (16, 8);
\filldraw[fill=black, draw=black] (17,7) rectangle (18, 8);
\filldraw[fill=black, draw=black] (17,8) rectangle (18, 9);
\filldraw[fill=black, draw=black] (18,8) rectangle (19, 9);
\filldraw[fill=black, draw=black] (15,9) rectangle (16, 10);
\filldraw[fill=black, draw=black] (18,9) rectangle (19, 10 );
\filldraw[fill=black, draw=black] (19,9) rectangle (20, 10);
\filldraw[fill=black, draw=black] (16,10) rectangle (17, 11);
\filldraw[fill=black, draw=black] (18,10) rectangle (19, 11);
\filldraw[fill=black, draw=black] (19,10) rectangle (20, 11);
\filldraw[fill=black, draw=black] (19,11) rectangle (20, 12);
\filldraw[fill=black, draw=black] (15,12) rectangle (16, 13);
\filldraw[fill=black, draw=black] (17,12) rectangle (18, 13);
\filldraw[fill=black, draw=black] (17,13) rectangle (18, 14);
\filldraw[fill=black, draw=black] (18,13) rectangle (19, 14);
\filldraw[fill=black, draw=black] (15,14) rectangle (16, 15);
\filldraw[fill=black, draw=black] (18,14) rectangle (19, 15);
\filldraw[fill=black, draw=black] (19,14) rectangle (20, 15);
\filldraw[fill=black, draw=black] (1,15) rectangle (2, 16);
\filldraw[fill=black, draw=black] (3,15) rectangle (4, 16);
\filldraw[fill=black, draw=black] (4,15) rectangle (5, 16);
\filldraw[fill=black, draw=black] (4,16) rectangle (5, 17);
\filldraw[fill=black, draw=black] (0,17) rectangle (1, 18);
\filldraw[fill=black, draw=black] (2,17) rectangle (3, 18);
\filldraw[fill=black, draw=black] (2,18) rectangle (3, 19);
\filldraw[fill=black, draw=black] (3,18) rectangle (4, 19);
\filldraw[fill=black, draw=black] (0,19) rectangle (1, 20);
\filldraw[fill=black, draw=black] (3,19) rectangle (4, 20);
\filldraw[fill=black, draw=black] (4,19) rectangle (5, 20);
\filldraw[fill=black, draw=black] (6,15) rectangle (7, 16);
\filldraw[fill=black, draw=black] (8,15) rectangle (9, 16);
\filldraw[fill=black, draw=black] (9,15) rectangle (10, 16);
\filldraw[fill=black, draw=black] (9,16) rectangle (10, 17);
\filldraw[fill=black, draw=black] (5,17) rectangle (6, 18);
\filldraw[fill=black, draw=black] (7,17) rectangle (8, 18);
\filldraw[fill=black, draw=black] (7,18) rectangle (8, 19);
\filldraw[fill=black, draw=black] (8,18) rectangle (9, 19);
\filldraw[fill=black, draw=black] (5,19) rectangle (6, 20);
\filldraw[fill=black, draw=black] (8,19) rectangle (9, 20);
\filldraw[fill=black, draw=black] (9,19) rectangle (10, 20);
\filldraw[fill=black, draw=black] (16,15) rectangle (17, 16);
\filldraw[fill=black, draw=black] (18,15) rectangle (19, 16);
\filldraw[fill=black, draw=black] (19,15) rectangle (20, 16);
\filldraw[fill=black, draw=black] (19,16) rectangle (20, 17);
\filldraw[fill=black, draw=black] (15,17) rectangle (16, 18);
\filldraw[fill=black, draw=black] (17,17) rectangle (18, 18);
\filldraw[fill=black, draw=black] (17,18) rectangle (18, 19);
\filldraw[fill=black, draw=black] (18,18) rectangle (19, 19);
\filldraw[fill=black, draw=black] (15,19) rectangle (16, 20);
\filldraw[fill=black, draw=black] (18,19) rectangle (19, 20);
\filldraw[fill=black, draw=black] (19,19) rectangle (20, 20);
\end{tikzpicture}
\caption{The set ${\cal W}_2$, constructed based on the above patterns ${\cal A}_1$ and ${\cal A}_2$, that
can also be viewed as a $20$-pattern} \label{fig:W2}
\end{center}
\end{figure}
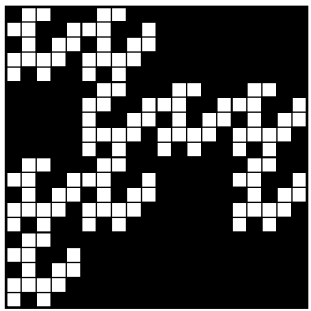

\begin{figure}[hhhh]
\begin{center}
\includegraphics[scale=0.4]{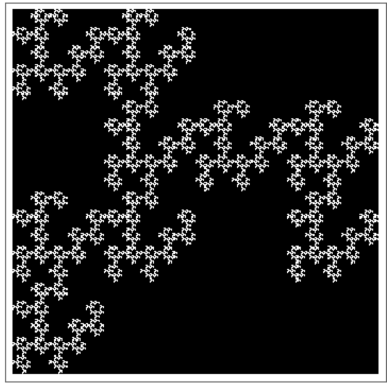}
\caption{A prefractal of the labyrinth fractal defined by a sequence of patterns that starts with the first theree patterns are the patterns ${\cal A}_1, {\cal A}_2, {\cal A}_3$ from Figure 1, and the fourth is ${\cal A}_1$}
\end{center}
\end{figure}

\par \noindent{\bf Remark.} Any labyrinth pattern is a labyrinth set, and any mixed labyrinth set can be 
seen as a labyrinth pattern. We use two distinct notions, in order to make it clearer that in general the 
construction of a (mixed) labyrinth set is based on a sequence of labyrinth patterns.



\section{Topological properties of mixed labyrinth fractals}\label{sec:Topological properties of mixed labyrinth fractals}
\begin{lemma}\label{lemma:Prop123} Let $\{{\cal A}_k\}_{k=1}^{\infty}$ be a sequence of non-empty patterns,  
$m_k\ge 3$, and $n\ge 1$. If ${\cal A}_1,\dots{\cal A}_n$ are labyrinth patterns,  then ${\cal W}_n$ is an $m(n)\times m(n)$-labyrinth set, for all $n \ge 1$, where 
$m(n)=\prod_{k=1}^n m_k$.
\end{lemma}
\begin{proof} The proof is almost literally the same as in the case of self-similar labyrinth fractals 
\cite[Lemma 2]{laby_4x4}. It only differs in two points in the second half of the proof: here  
instead of using the fact that $\mathcal{G}({\cal W}_{1})$ is a 
tree, we use the argument that $\mathcal{G}({\cal A}_{n})$ is a 
tree, and instead of using Property 2 for $\mathcal{G}({\cal W}_{1})$, we use Property 2 for $\mathcal{G}({\cal A}_{n})$.  

\end{proof}
 The limit set $L_{\infty}$ defined by a sequence $\{{\cal A}_k\}_{k=1}^{\infty}$ of labyrinth patterns is called \emph{mixed labyrinth fractal}. For $n\ge 1$, we define $\mathcal{G}({\cal B}_{n})\equiv (\mathcal{V}(\mathcal{G}({\cal B}_{n})),\mathcal{E}(\mathcal{G}({\cal B}_{n})))$ to be the graph whose vertices $\mathcal{V}(\mathcal{G}({\cal B}_{n}))$ are the black squares in ${\cal B}_{n}$, and whose edges $\mathcal{E}(\mathcal{G}({\cal B}_{n}))$ are the unordered pairs of black squares, that have a common side or a common corner. A \emph{border square of level $n$} is a square that lies in the right column, the left column, the top row, or the bottom row, of level $n$ respectively. 
\begin{lemma}\label{lemma:Steinhaus} Let $\{{\cal A}_k\}_{k=1}^{\infty}$ be a sequence of labyrinth patterns,  $m_k\ge 3$ and $n\ge 1$. From every black square in $\mathcal{G}({\cal B}_{n})$ there is a path in $\mathcal{G}({\cal B}_{n})$ to a black border square of level $n$ in $\mathcal{G}({\cal B}_{n})$.
\end{lemma}
\begin{proof} The proof  uses Lemma \ref{lemma:Prop123} and is literally the same as in the case of the labyrinth sets occuring in the construction of self-similar 
labyrinth fractals \cite[Lemma 2]{laby_4x4}.

\end{proof}
A function is a \emph{homeomorphism} if it is bijective, continuous, and its inverse is continuous. A topological space $X$ is an \emph{arc} if there is a homeomorphism $h$ from $[0,1]$ to $X$. We say $X$ is an arc between $h(0)$ and $h(1)$.
For the following two results we skip the proof, since the proofs given in the self-similar case \cite{laby_4x4} work also in the more general case of the 
mixed labyrinth sets. For more details and definitions we refer to the mentioned papers \cite{laby_4x4, laby_oigemoan}.

\begin{lemma}\label{lemma:border} Let $\{{\cal A}_k\}_{k=1}^{\infty}$ be a sequence of labyrinth patterns,  $m_k\ge 3$. If $x$ is a point in $([0,1]\times [0,1]) \setminus L_n$, then there is an arc $a \subseteq ([0,1]\times [0,1]) \setminus L_{n+1}$ between $x$ and a point in the boundary $fr([0,1]\times [0,1])$.
\end{lemma}

\begin{corollary}\label{corollary:border} Let $\{{\cal A}_k\}_{k=1}^{\infty}$ be a sequence of labyrinth patterns,  $m_k\ge 3$, and $n\ge 1$. If $x$ is a point in $([0,1]\times [0,1]) \setminus L_{\infty}$, then there is an arc $a \subseteq ([0,1]\times [0,1]) \setminus L_{\infty}$ between $x$ and a point in $fr([0,1]\times [0,1])$.
\end{corollary}


We remind that a \emph{continuum} is a compact connected Hausdorff space, and a \emph{dendrite} is a locally connected continuum that contains no simple closed curve. 
\begin{theorem}\label{theorem:dendrite} Let $\{{\cal A}_k\}_{k=1}^{\infty}$ be a sequence of labyrinth patterns,  
$m_k\ge 3$, for all $k \ge 1$. Then $L_{\infty}$ is a dendrite.
\end{theorem}
\begin{proof} $L_{\infty}$ is the intersection of the compact connected sets $L_n$, $n\ge 1$, 
and thus it is connected. One can easily check that for any $\epsilon>0$, there is an 
$n\ge 1$, such that the diameter of $W_{n} \in {\cal W}_{n}$ is less than $\epsilon$ (e.g., by using the 
facts that $m(n)>2^n$ and that the diameter of any square in ${\cal W}_n$ is $\frac{\sqrt{2}}{m(n)}$). Thus, for any $\epsilon>0$, $L_{\infty}$ is the 
finite union of connected sets of diameter less than $\epsilon$, by the definition of ${\cal W}_{n}$ (in Equation \ref{eq:W_n}). The Hahn-Mazurkiewicz-Sierpi\'nski 
Theorem \cite[Theorem~2, p.256]{Kuratowski} yields that $L_{\infty}$ is locally 
connected. As in the self-similar case \cite{laby_4x4,laby_oigemoan} one can show, 
by using the Jordan Curve Theorem and Corollary~\ref{corollary:border} that $L_{\infty}$ does not contain any simple closed curve.

\end{proof}
{\bf Remark.} Between any pair of points $x\neq y$ in $L_{\infty}$ there is a unique arc \cite[Corollary~2, p. 301]{Kuratowski}. 
\section{Paths in mixed labyrinth sets and in mixed labyrinth fractals}\label{sec:Paths}

In this section all patterns in the sequence $({\cal A}_k)_{k\ge 1}$ used in the iterative construction 
of ${\cal W}_n$ are labyrinth patterns.

We call a path in $\mathcal{G}({\cal W}_{n})$ a $\A$\emph{-path} if it leads from the top to 
the bottom exit of $W_n$. 
The $\B,\C,\D,\E$, and $\F$\emph{-paths} lead from left to right, top to right, right to bottom, bottom to left, and left to 
top exits, respectively. We denote by $\A(n),\B(n),\C(n),\D(n),\E(n)$, and $\F(n)$ the 
length of the respective path in 
$\mathcal{G}({\cal W}_{n})$, for $n\ge 1,$ and by
 $\A_k,\B_k,\C_k,\D_k,\E_k$, and $\F_k$ the length of the respective path in $\mathcal{G}({\cal A}_k)$, 
for $k\ge 1$. By the length of such a path we mean the number of squares in the path.
Of course, for $n=k=1$ the two path lengths coincide, i.e., $\A(1)=\A_1, \dots, \F(1)=\F_1$. 

\begin{proposition}\label{prop:recursion}
 There exist non-negative $6 \times 6$-matrices $M_k$, $k = 1,2,\dots$, such that
 \begin{equation}\label{eq:matrix_k}
 \left(
 \begin{array}{l}
\A_k \\  
\B_k \\  
\C_k \\  
\D_k \\  
\E_k \\  
\F_k \\  
\end{array}
\right)
=M_k \cdot \left(
\begin{array}{l}
1 \\  
1 \\  
1 \\  
1 \\  
1 \\  
1 \\  
\end{array}\right),
\end{equation}
and for $M(n)=M_1 \cdot M_2 \cdot \dots \cdot M_n$, for all $n \ge 1$, the element in row 
$x$ and column $y$ of $M(n)$ is the number of $y$-squares in the $x$-path in $\mathcal{G}({\cal W}_n)$.
Furthermore,
\begin{equation} \label{eq:recursion}
 \left(
 \begin{array}{l}
\A(n) \\  
\B(n) \\  
\C(n) \\  
\D(n) \\  
\E(n) \\  
\F(n) \\  
\end{array}
 \right)
=M(n) \cdot \left(
\begin{array}{l}
1 \\  
1 \\  
1 \\  
1 \\  
1 \\  
1 \\  
\end{array}\right).
\end{equation}
 
\end{proposition}

\begin{proof} We explain how the path between all possible pairs of exits can be constructed. In order to show the 
idea of this construction, we start, e.g.,                                    
with a path between the right and the bottom exit, 
as shown in Figure~\ref{fig:W2_path}.
We note that the construction described below works for all mixed labyrinth 
fractals. 

First, we find the path between the right and the bottom exit of ${\cal W}_{1}$, 
(or, equivalently ${\cal A}_{1}$) shown in Figure \ref{fig:Squares of type $C$ and $D$.}). Then we denote each white square in the path 
according to its neighbours within the path: 
if it has a top and a bottom neighbour it is called $\A$-\emph{square}  
(with respect to the path), and
it is called $\B,\C,\D,\E$, and $\F$-\emph{square} if its neighbours are at left-right,
top-right, 
right-bottom, bottom-left, and left-top, respectively. 
If the white square is an exit, it is supposed to have a neighbour outside the 
side of the exit. A bottom exit, e.g., is supposed to have a neighbour below, outside the bottom, additionally to its inside neighbour.
 We repeat this procedure for all possible paths between two exits in $\mathcal{G}({\cal W}_{1})$, as shown in 
 Figure \ref{fig:Squares of type $A$ and $B$.}, \ref{fig:Squares of type $C$ and $D$.}, and \ref{fig:Squares of type $E$ and $F$.}.
\begin{figure}[hhhh]
\begin{center}
\includegraphics[scale=1]{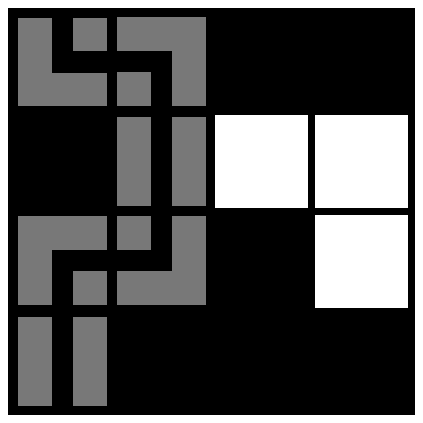}~~
\includegraphics[scale=1]{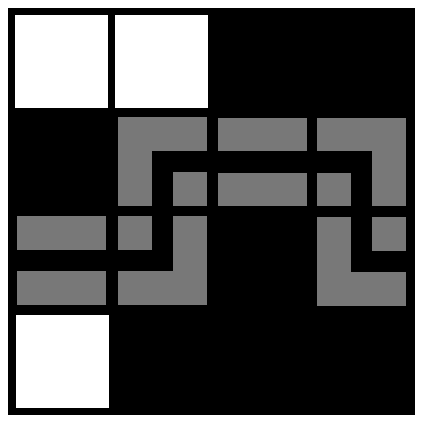}
\caption{Paths from top to bottom and from left to right exit of ${\cal A}_{1}$}\label{fig:Squares of type $A$ and $B$.}
\end{center}
\end{figure}
\begin{figure}[hhhh]
\begin{center}
\includegraphics[scale=1]{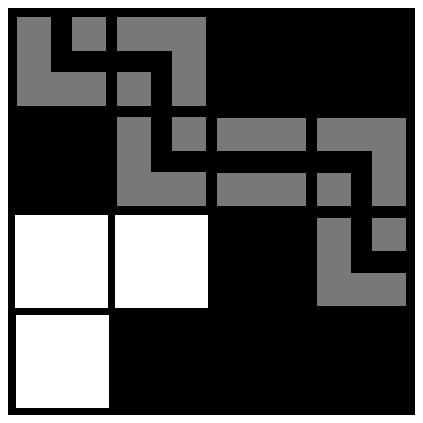}~~
\includegraphics[scale=1]{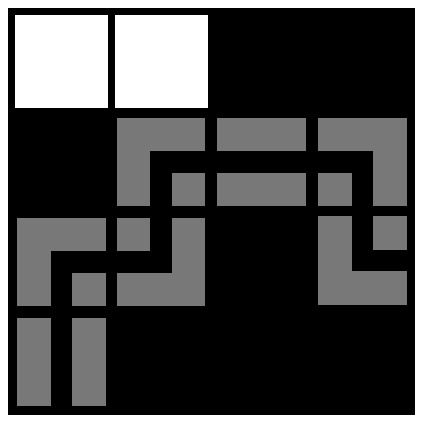}
\caption{Paths from top to right and from bottom to right exit of ${\cal A}_{1}$}\label{fig:Squares of type $C$ and $D$.}
\end{center}
\end{figure}
\begin{figure}[hhhh]
\begin{center}
\includegraphics[scale=1]{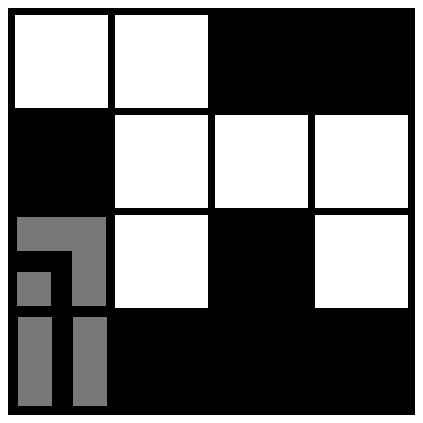}~~
\includegraphics[scale=1]{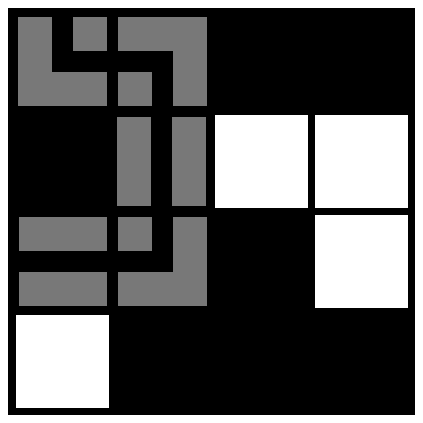}
\caption{Paths from left to bottom and top to left exit of ${\cal A}_{1}$}\label{fig:Squares of type $E$ and $F$.}
\end{center}
\end{figure}
In order to obtain the $\D$-path in $\mathcal{G}({\cal W}_{2})$, which is shown in Figure~\ref{fig:W2_path}, we replace each $\A$-square of 
the path in $\mathcal{G}({\cal W}_{1})$ with the $\A$-path in $\mathcal{G}({\mathcal A}_{2})$, which is shown in 
Figure~\ref{fig:A2_paths}.
Analogously, we do this for the other marked white squares, such that the path between 
the right and the bottom exit of ${\cal W}_{2}$ (shown in Figure~\ref{fig:W2_path}) arises. In general, for any pair of exits and $n\ge 1$, we replace each marked white 
square in the path of $\mathcal{G}({\cal W}_{n})$ with its corresponding path in $\mathcal{G}({\cal A}_{n+1})$ and obtain the path of 
$\mathcal{G}({\cal W}_{n+1})$. We define the matrix $M_k$, $k \ge 1$, that occurs in 
Equation~\ref{eq:matrix_k} in the following way: 
the columns of $M_k$ from left to right and the rows of $M_k$ from top to bottom correspond to $\A,\B,\C,\D,\E$, and $\F$, 
and the element in row 
$x$ and column $y$ of $M_k$ is the number of $y$-squares in the $x$-path in $\mathcal{G}({\cal A}_k)$. 
 One can easily check that the matrix multiplication reflects the substitution of paths. 
 The Equation \ref{eq:recursion} can then be shown by induction. 

\end{proof}

\begin{figure}[h!]
\begin{center}
\begin{tikzpicture}[scale=.41]

\draw[line width=2pt] (0,0) rectangle (10,10);

\filldraw[fill=gray!50, draw= black] (4,0) rectangle (6,2);
\filldraw[fill=gray!50, draw= black] (4,2) rectangle (6,4);
\filldraw[fill=gray!50, draw= black] (2,2) rectangle (4,4);
\filldraw[fill=gray!50, draw= black] (2,4) rectangle (4,6);
\filldraw[fill=gray!50, draw= black] (2,6) rectangle (4,8);
\filldraw[fill=gray!50, draw= black] (2,8) rectangle (4,10);
\filldraw[fill=gray!50, draw= black] (4,8) rectangle (6,10);
\draw[line width=6 pt, color=gray] (5, 0) -- (5, 3.15);
\draw[line width=6 pt, color=gray] (2.78,3) -- (5.25,3);
\draw[line width=6 pt, color=gray] (3, 2.82) -- (3, 9);
\draw[line width=6 pt, color=gray] (2.75, 9) -- (5.25, 9);
\draw[line width=6 pt, color=gray] (5, 9) -- (5, 10);
\draw[line width=2pt] (2, 0) -- (2,10);
\draw[line width=2pt] (4, 0) -- (4,10);
\draw[line width=2pt] (6, 0) -- (6,10);
\draw[line width=2pt] (8, 0) -- (8,10);
\draw[line width=2pt] (0, 2) -- (10,2);
\draw[line width=2pt] (0, 4) -- (10,4);
\draw[line width=2pt] (0, 6) -- (10,6);
\draw[line width=2pt] (0, 8) -- (10,8);
\filldraw[fill=black, draw=black] (2,0) rectangle (4, 2);
\filldraw[fill=black, draw=black] (6,0) rectangle (8, 2);
\filldraw[fill=black, draw=black] (8,0) rectangle (10, 2);
\filldraw[fill=black, draw=black] (8,2) rectangle (10, 4);
\filldraw[fill=black, draw=black] (0,4) rectangle (2, 6);
\filldraw[fill=black, draw=black] (4,4) rectangle (6, 6);
\filldraw[fill=black, draw=black] (4,6) rectangle (6, 8);
\filldraw[fill=black, draw=black] (6,6) rectangle (8, 8);
\filldraw[fill=black, draw=black] (0,8) rectangle (2, 10);
\filldraw[fill=black, draw=black] (6,8) rectangle (8, 10 );
\filldraw[fill=black, draw=black] (8,8) rectangle (10, 10);

\draw[line width=2pt] (11,0) rectangle (21,10);

\filldraw[fill=gray!50, draw= black] (11,6) rectangle (13,8);
\filldraw[fill=gray!50, draw= black] (13,6) rectangle (15,8);
\filldraw[fill=gray!50, draw= black] (13,4) rectangle (15,6);
\filldraw[fill=gray!50, draw= black] (13,2) rectangle (15,4);
\filldraw[fill=gray!50, draw= black] (15,2) rectangle (17,4);
\filldraw[fill=gray!50, draw= black] (17,2) rectangle (19,4);
\filldraw[fill=gray!50, draw= black] (17,4) rectangle (19,6);
\filldraw[fill=gray!50, draw= black] (19,4) rectangle (21,6);
\filldraw[fill=gray!50, draw= black] (19,6) rectangle (21,8);
\draw[line width=6 pt, color=gray] (11,7) -- (14.24,7 );
\draw[line width=6 pt, color=gray] (14, 2.82) -- (14, 7);
\draw[line width=6 pt, color=gray] (13.78,3) -- (18.35,3);
\draw[line width=6 pt, color=gray] (18.1,3 ) -- (18.1,5 );
\draw[line width=6 pt, color=gray] (17.85,5 ) -- (20.2,5 );
\draw[line width=6 pt, color=gray] (20,4.82 ) -- (20,7.18 );
\draw[line width=6 pt, color=gray] (19.8,7 ) -- (21,7 );
\draw[line width=2pt] (13, 0) -- (13,10);
\draw[line width=2pt] (15, 0) -- (15,10);
\draw[line width=2pt] (17, 0) -- (17,10);
\draw[line width=2pt] (19, 0) -- (19,10);
\draw[line width=2pt] (11, 2) -- (21,2);
\draw[line width=2pt] (11, 4) -- (21,4);
\draw[line width=2pt] (11, 6) -- (21,6);
\draw[line width=2pt] (11, 8) -- (21,8);
\filldraw[fill=black, draw=black] (13,0) rectangle (15, 2);
\filldraw[fill=black, draw=black] (17,0) rectangle (19, 2);
\filldraw[fill=black, draw=black] (19,0) rectangle (21, 2);
\filldraw[fill=black, draw=black] (19,2) rectangle (21, 4);
\filldraw[fill=black, draw=black] (11,4) rectangle (13, 6);
\filldraw[fill=black, draw=black] (15,4) rectangle (17, 6);
\filldraw[fill=black, draw=black] (15,6) rectangle (17, 8);
\filldraw[fill=black, draw=black] (17,6) rectangle (19, 8);
\filldraw[fill=black, draw=black] (11,8) rectangle (13, 10);
\filldraw[fill=black, draw=black] (17,8) rectangle (19, 10 );
\filldraw[fill=black, draw=black] (19,8) rectangle (21, 10);
\end{tikzpicture}
\caption{Paths from bottom to top and from left to right exit of $\mathcal{A}_2$}
\label{fig:A2_paths}
\end{center}
\end{figure}
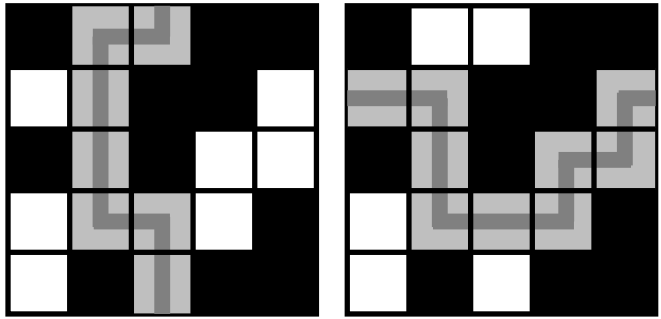

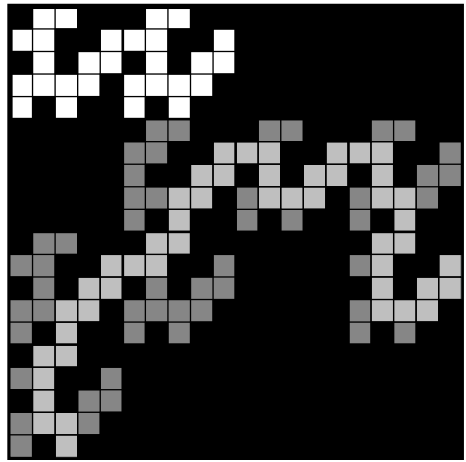
\begin{figure}[h!]
\begin{center}
\begin{tikzpicture}[scale=.3]
\draw[line width=2pt] (0,0) rectangle (20,20);
\filldraw[fill=gray!95, draw= black] (0,0) rectangle (5,5);
\filldraw[fill=gray!95, draw= black] (0,5) rectangle (5,10);
\filldraw[fill=gray!95, draw= black] (5,5) rectangle (10,10);
\filldraw[fill=gray!95, draw= black] (5,10) rectangle (10,15);
\filldraw[fill=gray!95, draw= black] (10,10) rectangle (15,15);
\filldraw[fill=gray!95, draw= black] (15,10) rectangle (20,15);
\filldraw[fill=gray!95, draw= black] (15,5) rectangle (20,10);
\filldraw[fill=gray!50, draw= black] (2,0) rectangle (3,1);
\filldraw[fill=gray!50, draw= black] (2,1) rectangle (3,2);
\filldraw[fill=gray!50, draw= black] (1,1) rectangle (2,2);
\filldraw[fill=gray!50, draw= black] (1,2) rectangle (2,3);
\filldraw[fill=gray!50, draw= black] (1,3) rectangle (2,4);
\filldraw[fill=gray!50, draw= black] (1,4) rectangle (2,5);
\filldraw[fill=gray!50, draw= black] (2,4) rectangle (3,5);
\filldraw[fill=gray!50, draw= black] (2,5) rectangle (3,6);
\filldraw[fill=gray!50, draw= black] (2,6) rectangle (3,7);
\filldraw[fill=gray!50, draw= black] (3,6) rectangle (4,7);
\filldraw[fill=gray!50, draw= black] (3,7) rectangle (4,8);
\filldraw[fill=gray!50, draw= black] (4,7) rectangle (5,8);
\filldraw[fill=gray!50, draw= black] (4,8) rectangle (5,9);
\filldraw[fill=gray!50, draw= black] (5,8) rectangle (6,9);
\filldraw[fill=gray!50, draw= black] (6,8) rectangle (7,9);
\filldraw[fill=gray!50, draw= black] (6,9) rectangle (7,10);
\filldraw[fill=gray!50, draw= black] (7,9) rectangle (8,10);
\filldraw[fill=gray!50, draw= black] (7,10) rectangle (8,11);
\filldraw[fill=gray!50, draw= black] (7,11) rectangle (8,12);
\filldraw[fill=gray!50, draw= black] (8,11) rectangle (9,12);
\filldraw[fill=gray!50, draw= black] (8,12) rectangle (9,13);
\filldraw[fill=gray!50, draw= black] (9,12) rectangle (10,13);
\filldraw[fill=gray!50, draw= black] (9,13) rectangle (10,14);
\filldraw[fill=gray!50, draw= black] (10,13) rectangle (11,14);
\filldraw[fill=gray!50, draw= black] (11,13) rectangle (12,14);
\filldraw[fill=gray!50, draw= black] (11,12) rectangle (12,13);
\filldraw[fill=gray!50, draw= black] (11,11) rectangle (12,12);
\filldraw[fill=gray!50, draw= black] (12,11) rectangle (13,12);
\filldraw[fill=gray!50, draw= black] (13,11) rectangle (14,12);
\filldraw[fill=gray!50, draw= black] (13,12) rectangle (14,13);
\filldraw[fill=gray!50, draw= black] (14,12) rectangle (15,13);
\filldraw[fill=gray!50, draw= black] (14,13) rectangle (15,14);
\filldraw[fill=gray!50, draw= black] (15,13) rectangle (16,14);
\filldraw[fill=gray!50, draw= black] (16,13) rectangle (17,14);
\filldraw[fill=gray!50, draw= black] (16,12) rectangle (17,13);
\filldraw[fill=gray!50, draw= black] (16,11) rectangle (17,12);
\filldraw[fill=gray!50, draw= black] (17,11) rectangle (18,12);
\filldraw[fill=gray!50, draw= black] (17,10) rectangle (18,11);
\filldraw[fill=gray!50, draw= black] (17,9) rectangle (18,10);
\filldraw[fill=gray!50, draw= black] (16,9) rectangle (17,10);
\filldraw[fill=gray!50, draw= black] (16,8) rectangle (17,9);
\filldraw[fill=gray!50, draw= black] (16,7) rectangle (17,8);
\filldraw[fill=gray!50, draw= black] (16,6) rectangle (17,7);
\filldraw[fill=gray!50, draw= black] (17,6) rectangle (18,7);
\filldraw[fill=gray!50, draw= black] (18,6) rectangle (19,7);
\filldraw[fill=gray!50, draw= black] (18,7) rectangle (19,8);
\filldraw[fill=gray!50, draw= black] (19,7) rectangle (20,8);
\filldraw[fill=gray!50, draw= black] (19,8) rectangle (20,9);

\draw[line width=1pt] (5, 0) -- (5,20);
\draw[line width=1pt] (10, 0) -- (10,20);
\draw[line width=1pt] (15, 0) -- (15,20);
\draw[line width=1pt] (0, 5) -- (20,5);
\draw[line width=1pt] (0, 10) -- (20,10);
\draw[line width=1pt] (0, 15) -- (20,15);
\draw[line width=0.5pt] (1, 0) -- (1,20);
\draw[line width=0.5pt] (2, 0) -- (2,20);
\draw[line width=0.5pt] (3, 0) -- (3,20);
\draw[line width=0.5pt] (4, 0) -- (4,20);
\draw[line width=0.5pt] (6, 0) -- (6,20);
\draw[line width=0.5pt] (7, 0) -- (7,20);
\draw[line width=0.5pt] (8, 0) -- (8,20);
\draw[line width=0.5pt] (9, 0) -- (9,20);
\draw[line width=0.5pt] (11, 0) -- (11,20);
\draw[line width=0.5pt] (12, 0) -- (12,20);
\draw[line width=0.5pt] (13, 0) -- (13,20);
\draw[line width=0.5pt] (14, 0) -- (14,20);
\draw[line width=0.5pt] (16, 0) -- (16,20);
\draw[line width=0.5pt] (17, 0) -- (17,20);
\draw[line width=0.5pt] (18, 0) -- (18,20);
\draw[line width=0.5pt] (19, 0) -- (19,20);
\draw[line width=0.5pt] (0, 1) -- (20,1);
\draw[line width=0.5pt] (0, 2) -- (20,2);
\draw[line width=0.5pt] (0, 3) -- (20,3);
\draw[line width=0.5pt] (0, 4) -- (20,4);
\draw[line width=0.5pt] (0, 6) -- (20,6);
\draw[line width=0.5pt] (0, 7) -- (20,7);
\draw[line width=0.5pt] (0, 8) -- (20,8);
\draw[line width=0.5pt] (0, 9) -- (20,9);
\draw[line width=0.5pt] (0, 11) -- (20,11);
\draw[line width=0.5pt] (0, 12) -- (20,12);
\draw[line width=0.5pt] (0, 13) -- (20,13);
\draw[line width=0.5pt] (0, 14) -- (20,14);
\draw[line width=0.5pt] (0, 16) -- (20,16);
\draw[line width=0.5pt] (0, 17) -- (20,17);
\draw[line width=0.5pt] (0, 18) -- (20,18);
\draw[line width=0.5pt] (0, 19) -- (20,19);
\filldraw[fill=black, draw=black] (5,0) rectangle (10, 5);
\filldraw[fill=black, draw=black] (10,0) rectangle (15, 5);
\filldraw[fill=black, draw=black] (15,0) rectangle (20, 5);
\filldraw[fill=black, draw=black] (10,5) rectangle (15, 10);
\filldraw[fill=black, draw=black] (0,10) rectangle (5, 15);
\filldraw[fill=black, draw=black] (10,15) rectangle (15, 20);
\filldraw[fill=black, draw=black] (15,15) rectangle (20, 20);

\filldraw[fill=black, draw=black] (1,0) rectangle (2, 1);
\filldraw[fill=black, draw=black] (3,0) rectangle (4, 1);
\filldraw[fill=black, draw=black] (4,0) rectangle (5, 1);
\filldraw[fill=black, draw=black] (4,1) rectangle (5, 2);
\filldraw[fill=black, draw=black] (0,2) rectangle (1, 3);
\filldraw[fill=black, draw=black] (2,2) rectangle (3, 3);
\filldraw[fill=black, draw=black] (2,3) rectangle (3, 4);
\filldraw[fill=black, draw=black] (3,3) rectangle (4, 4);
\filldraw[fill=black, draw=black] (0,4) rectangle (1, 5);
\filldraw[fill=black, draw=black] (3,4) rectangle (4, 5 );
\filldraw[fill=black, draw=black] (4,4) rectangle (5, 5);
\filldraw[fill=black, draw=black] (1,5) rectangle (2, 6);
\filldraw[fill=black, draw=black] (3,5) rectangle (4, 6);
\filldraw[fill=black, draw=black] (4,5) rectangle (5, 6);
\filldraw[fill=black, draw=black] (4,6) rectangle (5, 7);
\filldraw[fill=black, draw=black] (0,7) rectangle (1, 8);
\filldraw[fill=black, draw=black] (2,7) rectangle (3, 8);
\filldraw[fill=black, draw=black] (2,8) rectangle (3, 9);
\filldraw[fill=black, draw=black] (3,8) rectangle (4, 9);
\filldraw[fill=black, draw=black] (0,9) rectangle (1, 10);
\filldraw[fill=black, draw=black] (3,9) rectangle (4, 10);
\filldraw[fill=black, draw=black] (4,9) rectangle (5, 10);
\filldraw[fill=black, draw=black] (6,5) rectangle (7, 6);
\filldraw[fill=black, draw=black] (8,5) rectangle (9, 6);
\filldraw[fill=black, draw=black] (9,5) rectangle (10, 6);
\filldraw[fill=black, draw=black] (9,6) rectangle (10, 7);
\filldraw[fill=black, draw=black] (5,7) rectangle (6, 8);
\filldraw[fill=black, draw=black] (7,7) rectangle (8, 8);
\filldraw[fill=black, draw=black] (7,8) rectangle (8, 9);
\filldraw[fill=black, draw=black] (8,8) rectangle (9, 9);
\filldraw[fill=black, draw=black] (5,9) rectangle (6, 10);
\filldraw[fill=black, draw=black] (8,9) rectangle (9, 10);
\filldraw[fill=black, draw=black] (9,9) rectangle (10, 10);
\filldraw[fill=black, draw=black] (6,10) rectangle (7, 11);
\filldraw[fill=black, draw=black] (8,10) rectangle (9, 11);
\filldraw[fill=black, draw=black] (9,10) rectangle (10, 11);
\filldraw[fill=black, draw=black] (9,11) rectangle (10, 12);
\filldraw[fill=black, draw=black] (7,12) rectangle (6, 13);
\filldraw[fill=black, draw=black] (7,12) rectangle (8, 13);
\filldraw[fill=black, draw=black] (7,13) rectangle (8, 14);
\filldraw[fill=black, draw=black] (8,13) rectangle (9, 14);
\filldraw[fill=black, draw=black] (5,14) rectangle (6, 15);
\filldraw[fill=black, draw=black] (8,14) rectangle (9, 15 );
\filldraw[fill=black, draw=black] (9,14) rectangle (10, 15);
\filldraw[fill=black, draw=black] (11,10) rectangle (12, 11);
\filldraw[fill=black, draw=black] (13,10) rectangle (14, 11);
\filldraw[fill=black, draw=black] (14,10) rectangle (15, 11);
\filldraw[fill=black, draw=black] (14,11) rectangle (15, 12);
\filldraw[fill=black, draw=black] (10,12) rectangle (11, 13);
\filldraw[fill=black, draw=black] (12,12) rectangle (13, 13);
\filldraw[fill=black, draw=black] (12,13) rectangle (13, 14);
\filldraw[fill=black, draw=black] (13,13) rectangle (14, 14);
\filldraw[fill=black, draw=black] (10,14) rectangle (11, 15);
\filldraw[fill=black, draw=black] (13,14) rectangle (14, 15 );
\filldraw[fill=black, draw=black] (14,14) rectangle (15, 15);
\filldraw[fill=black, draw=black] (16,5) rectangle (17, 6);
\filldraw[fill=black, draw=black] (18,5) rectangle (19, 6);
\filldraw[fill=black, draw=black] (19,5) rectangle (20, 6);
\filldraw[fill=black, draw=black] (19,6) rectangle (20, 7);
\filldraw[fill=black, draw=black] (15,7) rectangle (16, 8);
\filldraw[fill=black, draw=black] (17,7) rectangle (18, 8);
\filldraw[fill=black, draw=black] (17,8) rectangle (18, 9);
\filldraw[fill=black, draw=black] (18,8) rectangle (19, 9);
\filldraw[fill=black, draw=black] (15,9) rectangle (16, 10);
\filldraw[fill=black, draw=black] (18,9) rectangle (19, 10 );
\filldraw[fill=black, draw=black] (19,9) rectangle (20, 10);
\filldraw[fill=black, draw=black] (16,10) rectangle (17, 11);
\filldraw[fill=black, draw=black] (18,10) rectangle (19, 11);
\filldraw[fill=black, draw=black] (19,10) rectangle (20, 11);
\filldraw[fill=black, draw=black] (19,11) rectangle (20, 12);
\filldraw[fill=black, draw=black] (15,12) rectangle (16, 13);
\filldraw[fill=black, draw=black] (17,12) rectangle (18, 13);
\filldraw[fill=black, draw=black] (17,13) rectangle (18, 14);
\filldraw[fill=black, draw=black] (18,13) rectangle (19, 14);
\filldraw[fill=black, draw=black] (15,14) rectangle (16, 15);
\filldraw[fill=black, draw=black] (18,14) rectangle (19, 15);
\filldraw[fill=black, draw=black] (19,14) rectangle (20, 15);
\filldraw[fill=black, draw=black] (1,15) rectangle (2, 16);
\filldraw[fill=black, draw=black] (3,15) rectangle (4, 16);
\filldraw[fill=black, draw=black] (4,15) rectangle (5, 16);
\filldraw[fill=black, draw=black] (4,16) rectangle (5, 17);
\filldraw[fill=black, draw=black] (0,17) rectangle (1, 18);
\filldraw[fill=black, draw=black] (2,17) rectangle (3, 18);
\filldraw[fill=black, draw=black] (2,18) rectangle (3, 19);
\filldraw[fill=black, draw=black] (3,18) rectangle (4, 19);
\filldraw[fill=black, draw=black] (0,19) rectangle (1, 20);
\filldraw[fill=black, draw=black] (3,19) rectangle (4, 20);
\filldraw[fill=black, draw=black] (4,19) rectangle (5, 20);
\filldraw[fill=black, draw=black] (6,15) rectangle (7, 16);
\filldraw[fill=black, draw=black] (8,15) rectangle (9, 16);
\filldraw[fill=black, draw=black] (9,15) rectangle (10, 16);
\filldraw[fill=black, draw=black] (9,16) rectangle (10, 17);
\filldraw[fill=black, draw=black] (5,17) rectangle (6, 18);
\filldraw[fill=black, draw=black] (7,17) rectangle (8, 18);
\filldraw[fill=black, draw=black] (7,18) rectangle (8, 19);
\filldraw[fill=black, draw=black] (8,18) rectangle (9, 19);
\filldraw[fill=black, draw=black] (5,19) rectangle (6, 20);
\filldraw[fill=black, draw=black] (8,19) rectangle (9, 20);
\filldraw[fill=black, draw=black] (9,19) rectangle (10, 20);

\end{tikzpicture}
\caption{The set  $\mathcal{W}_2$ constructed with the patterns 
$\mathcal{A}_1$ and $\mathcal{A}_2$ shown in Figure \ref{fig:A1A2}, and the path from the bottom 
to the right exit of $ \mathcal{W}_2$ (in lighter gray)} 
\label{fig:W2_path}
\end{center}
\end{figure}

We note that in the above example, 
\[
M_1=\left(
\begin{array}{llllll}
2 & 0 & 1 & 1 & 1 & 1\\
0 & 2 & 1 & 1 & 1 & 1\\
0 & 1 & 3 & 0 & 2 & 0\\
1 & 1 & 1 & 2 & 1 & 1\\
1 & 0 & 0 & 0 & 1 & 0\\
1 & 1 & 1 & 0 & 1 & 1\\                                                                                         
\end{array}\right),~~\mbox{and}~~ 
M_2=\left(
\begin{array}{llllll}
3 & 0 & 1 & 1 & 1 & 1\\
1 & 2 & 1 & 2 & 1 & 2\\
2 & 1 & 1 & 3 & 0 & 3\\
1 & 0 & 0 & 3 & 0 & 2\\
2 & 1 & 1 & 0 & 2 & 0\\
0 & 1 & 0 & 1 & 0 & 2\\
\end{array}\right).
\]

For the above matrices we obtain $M_1\cdot M_2 =\left(
\begin{array}{llllll}
11 & 3 & 4 &  ~9 & 4 &  ~9\\
 ~7 & 7 & 4 & 11 & 4 & 11\\
11 & 7 & 6 & 11 & 5 & 11\\
10 & 5 & 4 & 13 & 4 & 12\\
 ~5 & 1 & 2 &  ~1 & 3 &  ~1\\
 ~8 & 5 & 4 &  ~7 & 4 &  ~8\\
\end{array}
\right),
$
and one can check (see also Figure~\ref{fig:W2_path}) that for this matrix the element in row $x$ and column $y$ is the number of $y$-squares in the $x$-path in 
$\mathcal{G}({\cal W}_{2})$.

We call the matrix $M_k$ in Proposition~\ref{prop:recursion} \emph{the path matrix of the labyrinth pattern} ${\cal A}_{k}$, $k=1,2,\dots$, 
and $M(n)$ the \emph{the path matrix of the (mixed) labyrinth set} ${\cal W}_{n}$, for $n=1,2,\dots$. 
\par For $n\ge 1$ and $W_1,W_2\in 
{\mathcal V}(\mathcal{G}({\cal W}_{n}))$, let $p_n(W_1,W_2)$ be the path in $\mathcal{G}({\cal W}_{n})$ from $W_1$ to $W_2$. 
Lemma~\ref{lemma:Construction} can be proven, as in the self-similar case \cite[Lemma~6]{laby_4x4} by using a theorem from the book of 
Kuratowski \cite[Theorem~3, par. 47, V, p. 181]{Kuratowski}.

\begin{lemma}\label{lemma:Construction}(Arc Construction) Let $a,b\in L_{\infty}$, where $a\neq b$. For all $n \ge 1$, there are $W_n(a),W_n(b)\in {\mathcal V}(\mathcal{G}({\cal W}_{n}))$ such that 
\begin{itemize}
\item[(a)]$W_1(a)\supseteq W_2(a)\supseteq\ldots$,
\item[(b)]$W_1(b)\supseteq W_2(b)\supseteq\ldots$,
\item[(c)]$\{a\}=\bigcap_{n=1}^{\infty}W_n(a)$,
\item[(d)]$\{b\}=\bigcap_{n=1}^{\infty}W_n(b)$.
\item[(e)]The set $\bigcap_{n=1}^{\infty}\left(\bigcup_{W\in p_n(W_n(a),W_n(b))} W\right)$ is an arc between $a$ and $b$. 
\end{itemize}
\end{lemma}

Let $T_n\in {\cal W}_{n}$ be the top exit of ${\cal W}_{n}$, for $n\ge 1$. The \emph{top exit of} 
$L_{\infty}$ is $\bigcap_{n=1}^{\infty}T_n$. The other exits of $L_{\infty}$ are defined analogously. We note that 
Property~\ref{prop2} yields that $(x,1),(x,0)\in L_{\infty}$ if and only if $(x,1)$ is the top exit of $L_{\infty}$ and $(x,0)$ 
is the bottom exit of $L_{\infty}$. For the left and the right exits the analogue statement holds. 

The proof of the following proposition is analogous to that in the self-similar case \cite[Lemma~7]{laby_4x4}, taking into account that 
in the case of mixed labyrinth fractals the edgelength of a square of level $n$ is $\frac{1}{m(n)}$.
For the definitions of the parametrisation of a curve and its length we refer, e.g., to the mentioned paper \cite{laby_4x4}.

\begin{proposition}\label{lemma:m^n} Let $n,k\ge 1$, $\{W_1,\ldots,W_k\}$ be a (shortest) path between the exits $W_1$ and $W_k$ in 
$\mathcal{G}({\cal W}_{n})$,  $K_0=W_1 \cap fr([0,1]\times[0,1])$, $K_k=W_k \cap fr([0,1]\times[0,1])$, and $c$ be a curve in $L_n$ 
from a point of $K_0$ to a point of $K_k$. The length of any parametrisation of $c$ is at least $(k-1)/(2\cdot m(n))$.
\end{proposition}

Let $n\ge 1$, $W\in {\cal W}_{n}$, and $t$ be the intersection of $L_{\infty}$ with the top edge of $W$. 
Then we call $t$ the \emph{top exit} of $W$. Analogously we define the \emph{bottom exit}, the\emph{ left exit} and
the \emph{right exit} of $W$.
 We note that the uniqueness of each of these four exits is provided by the uniqueness of the four exits of a 
mixed labyrinth fractal and by the fact that each 
such set of the form 
$L_{\infty} \cap W$, where $W\in {\cal W}_{n}$, is a mixed labyrinth fractal scaled by the factor $m(n)$.
We note that we have now defined exits for 
three different types of objects, i.e., for ${\cal W}_{n}$ (and ${\cal A}_{k}$), for $L_{\infty}$, 
and for squares in ${\cal W}_{n}$.
\begin{proposition}\label{lemma:ArcSimilarity1} Let $e_1,e_2$ be two exits in $L_{\infty}$, and $W_n(e_1), 
W_n(e_2)$ be the exits in $\mathcal{G}({\cal W}_{n})$ of the same type as $e_1$ and $e_2$, respectively, 
for some $n\ge 1.$
If $a$ is the arc that connects $e_1$ and $e_2$ in $L_{\infty}$, $p$ is the path in 
$\mathcal{G}({\cal W}_{n})$ from $W_n(e_1)$ to $W_n(e_2)$, and $W\in {\cal W}_{n}$ is a $\A$-square 
with respect to $p$, then $W\cap a$ is an arc in $L_{\infty}$ between the top and the bottom exit of $W$. 
If $W$ is an other type of square, the corresponding analogue statement 
holds.
\end{proposition}
\begin{proof}
 Analogously to the self-similar case, the statement follows from Lemma 
\ref{lemma:Construction}. 
\end{proof}



By the construction of mixed labyrinth fractals we obtain the following result.
\begin{proposition} \label{prop:exits_coordinates}Let $\{{\cal A}_k\}_{k=1}^{\infty}$ be a sequence of 
labyrinth patterns,  $m_k\ge 3$, and we set $m(0):=1$. 
 \begin{itemize}
 \item[(a)] 
Let $t_1, t_2$ and $b_1, b_2$ be the Cartesian coordinates of the top exit and bottom exit, respectively, 
in $L_{\infty}$, and $x_{1,t}^{k}, x_{2,t}^{k}$ the Cartesian coordinates of the left lower vertex of 
the square that is the top exit in ${{\cal A}_k}$, for all $k\ge 1$. Then
\[
t_1=b_1=\sum_{k=1}^{\infty}\frac{x_{1,t}^k}{m(k-1)}, ~~t_2=1, ~~b_2=0.
 \]
 \item[(b)] 
Let ${l_1}, l_2$ and $r_1, r_2$ be the Cartesian coordinates of the left exit and right exit, respectively, 
in $L_{\infty}$, and $x_{1,l}^{k}, x_{2,l}^{k}$ the Cartesian coordinates of the left lower vertex of 
the square that is the left exit in ${{\cal A}_k}$, for all $k \ge 1$. Then
\[
l_2=r_2=\sum_{k=1}^{\infty}\frac{x_{2,l}^k}{ m(k-1)}, ~~l_1=0, ~~r_1=1.
 \]
 \end{itemize}
\end{proposition}
{\bf Remark.} In the self-similar case it was shown \cite[Lemma 4]{laby_oigemoan} that 
for all $n \ge 1$, each exit $e$ in $L_{\infty}$ lies in exactly one square $W_n(e)\in{\cal W}_n$.

The following counterexample shows that the above statement does not hold in general in the case of mixed labyrinth fractals.
Let, e.g., ${\cal A}_1$ and ${\cal A}_2$ be as shown in Figure \ref{fig:A1A2} \ref{fig:counterexample_1}, and ${\cal A}_k={\cal A}_2$, 
for all $k\ge 3$. One can check (e.g., with Proposition \ref{prop:exits_coordinates}) that the left exit of $L_{\infty}$
is the midpoint of the left side of the unit square, i.e., the point $(0, \frac{1}{2})$, 
and that this exit lies in two squares of ${\cal W}_1,$  and in exactly one square of ${\cal W}_n,$ for all $n\ge 2.$

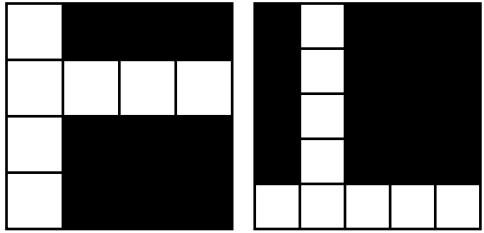
\begin{figure}[hhhh]
\begin{center}
\begin{tikzpicture}[scale=.30]
\draw[line width=1pt] (0,0) rectangle (10,10);
\draw[line width=1pt] (2.5, 0) -- (2.5,10);
\draw[line width=1pt] (5, 0) -- (5,10);
\draw[line width=1pt] (7.5, 0) -- (7.5,10);
\draw[line width=1pt] (0, 2.5) -- (10,2.5);
\draw[line width=1pt] (0, 5) -- (10,5);
\draw[line width=1pt] (0, 7.5) -- (10,7.5);
\filldraw[fill=black, draw=black] (2.5,0) rectangle (10, 5);
\filldraw[fill=black, draw=black] (2.5,7.5) rectangle (10, 10);
\draw[line width=1pt] (11,0) rectangle (21,10);
\draw[line width=1pt] (13, 0) -- (13,10);
\draw[line width=1pt] (15, 0) -- (15,10);
\draw[line width=1pt] (17, 0) -- (17,10);
\draw[line width=1pt] (19, 0) -- (19,10);
\draw[line width=1pt] (11, 2) -- (21,2);
\draw[line width=1pt] (11, 4) -- (21,4);
\draw[line width=1pt] (11, 6) -- (21,6);
\draw[line width=1pt] (11, 8) -- (21,8);
\filldraw[fill=black, draw=black] (11,2) rectangle (13, 10);
\filldraw[fill=black, draw=black] (15,2) rectangle (21, 10);
\end{tikzpicture}
\caption{Two labyrinth patterns, ${\cal A}_1$ (a $4$-pattern) and ${\cal A}_2$ (a $5$-pattern)}
\label{fig:counterexample_1}
\end{center}
\end{figure}

\begin{figure}[hhhh]
\begin{center}
\begin{tikzpicture}[scale=.30]
\draw[line width=1pt] (0,0) rectangle (10,10);
\draw[line width=1pt] (2.5, 0) -- (2.5,10);
\draw[line width=1pt] (5, 0) -- (5,10);
\draw[line width=1pt] (7.5, 0) -- (7.5,10);
\draw[line width=1pt] (0, 2.5) -- (10,2.5);
\draw[line width=1pt] (0, 5) -- (10,5);
\draw[line width=1pt] (0, 7.5) -- (10,7.5);
\filldraw[fill=black, draw=black] (2.5,0) rectangle (10, 5);
\filldraw[fill=black, draw=black] (2.5,7.5) rectangle (10, 10);
\draw[line width=1pt] (15,0) rectangle (25,10);
\draw[line width=1pt] (17.5, 0) -- (17.5,10);
\draw[line width=1pt] (20, 0) -- (20,10);
\draw[line width=1pt] (22.5, 0) -- (22.5,10);
\draw[line width=1pt] (15, 2.5) -- (25,2.5);
\draw[line width=1pt] (15, 5) -- (25,5);
\draw[line width=1pt] (15, 7.5) -- (25,7.5);
\filldraw[fill=black, draw=black] (17.5,0) rectangle (25, 2.5);
\filldraw[fill=black, draw=black] (17.5,2.5) rectangle (20, 5);
\filldraw[fill=black, draw=black] (22.5,5) rectangle (25, 7.5);
\filldraw[fill=black, draw=black] (17.5,7.5) rectangle (25, 10);
\draw[line width=1pt] (30,0) rectangle (40,10);
\draw[line width=1pt] (32, 0) -- (32,10);
\draw[line width=1pt] (34, 0) -- (34,10);
\draw[line width=1pt] (36, 0) -- (36,10);
\draw[line width=1pt] (38, 0) -- (38,10);
\draw[line width=1pt] (30, 2) -- (40,2);
\draw[line width=1pt] (30, 4) -- (40,4);
\draw[line width=1pt] (30, 6) -- (40,6);
\draw[line width=1pt] (30, 8) -- (40,8);
\filldraw[fill=black, draw=black] (30,2) rectangle (32, 10);
\filldraw[fill=black, draw=black] (34,0) rectangle (38, 2);
\filldraw[fill=black, draw=black] (34,4) rectangle (40, 10);

\end{tikzpicture}
\caption{Three labyrinth patterns,  ${\cal A}_1$, ${\cal A}_2$, and ${\cal A}_3$}
\label{fig:counterexample_2}
\end{center}
\end{figure}
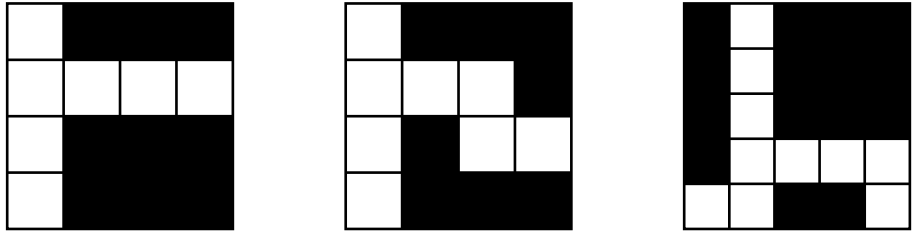

In Figure \ref{fig:counterexample_2} we show another counterexample. 
Let, e.g., ${\cal A}_1$, ${\cal A}_2$, and ${\cal A}_3$ be as shown in the figure, and ${\cal A}_k={\cal A}_3$, 
for all $k\ge 4$.
One can check that the left exit of $L_{\infty}$
is the point $(0, \frac{1}{2}+\frac{1}{16})$, 
and that this exit lies in exactly one square of ${\cal W}_1$, 
in two squares 
of ${\cal W}_2$, and in 
exactly one square of ${\cal W}_n$, for all $n\ge 3.$

\vspace{0.2cm}
For all $n \ge 1$, and $W \in {\cal W}_n$ let $L_{\infty}|W:= L_{\infty}\cap W$ and, 
for   $e \in \{t,b,l,r\}$, let $e(W)$ denote the top, bottom, left, or right exit of $W$, respectively.
\begin{proposition} With the above notations, we have: $L_{\infty}=\cup_{W \in {{\cal W}_n}}L_{\infty}|W$.
For all $n \ge 1$ and $W_1, W_2 \in {\cal W}_n$ there exists a translation $\phi$, 
such that $\phi (L_{\infty}|W_1)=L_{\infty}|W_2.$ Moreover, then we also have $\phi(e(W_1))=e(W_2)$, for $e\in\{t,b, l,r\}$.
\end{proposition}
\begin{proof}
 The above statement follows from the construction of labyrinth sets and the definition of a mixed labyrinth fractal.
\end{proof}

\begin{proposition}\label{prop:Boxcounting} If $a$ is an arc between the top and the bottom exit in $L_{\infty}$ then 
\[
\liminf_{n\rightarrow \infty} \frac{\log(\A(n))}{\sum_{k=1}^{n}\log (m_k)}= \underline{\dim}_B(a)\le\overline{\dim}_B(a)= \limsup_{n\rightarrow \infty} \frac{\log(\A(n))}{\sum_{k=1}^{n}\log (m_k)}.
\]

For the other pairs of exits, the analogue statement holds. 
\end{proposition}

\begin{proof}
 The above inequalities follow, e.g., by using an alternative definiton of the box counting 
dimension 
\cite[Definition 1.3.]{Falconer_book} and making use of the property that 
one can use, instead of $\delta \to 0$, appropriate sequences $(\delta_k)_{k\ge 0}$ 
for the computation of the box counting dimension (see the cited reference).
For $\delta_k=\frac{1}{m(k)}$, we have $\delta_k \le \frac{1}{3}\delta_{k-1}$, and one immediately gets the above formul\ae.
\end{proof}

\section{Blocked labyrinth patterns}\label{sec:Blocked}
An $m\times m$-labyrinth pattern ${\cal A}$ is called \emph{horizontally blocked} if the row (of squares) 
from the left to the right exit contains at least one black square. It is called \emph{vertically blocked} if the 
column (of squares) from the top to the bottom exit contains at least one black square. Anogously we define for any 
$n \ge 1$ a horizontally or vertically blocked labyrinth set of level $n$.
We note that there are no horizontally or 
vertically blocked $m\times m$-labyrinth patterns for $m<4$. As an example, the labyrinth patterns shown in Figure \ref{fig:A1A2} and \ref{fig:complement} 
are horizontally and vertically blocked, while those in Figure \ref{fig:counterexample_1} are not blocked.

\begin{conjecture}\label{conj:main result}
 Let $\{{\cal A}_k\}_{k=1}^{\infty}$ be a sequence of (both horizontally and vertically) blocked labyrinth patterns,  
 $m_k\ge 4$. 
 For any two points in the limit set $L_{\infty}$ the length of the arc $a\subset L_{\infty}$ that connects them is infinite and the set of all points, where no tangent to $a$ exists, is dense in $a$.
 \end{conjecture}

\section{Rectangular mixed labyrinth sets and fractals}\label{sec:Rectangular}
In order to define \emph{rectangular mixed labyrinth sets and fractals} we introduce the following function $P_R$, 
in analogy with $P_Q$ mentioned in Section \ref{sec:Construction}. 
Let $x,y,p,q\in [0,1]$ such that $R=[x,x+p]\times [y,y+q]\subseteq [0,1]\times [0,1]$. 
Then for any point $(z_x,z_y)\in[0,1]\times [0,1]$ we define the function
\[
P_R(z_x,z_y)=(p z_x+x,q z_y+y).
\]
For any integers $m,s\ge 1$ , let $S_{i,j;m,s}=\{(x,y)\mid \frac{i}{m}\le x \le \frac{i+1}{m} \mbox{ and } \frac{j}{s}\le y \le \frac{j+1}{s} \}$ and  
${\cal S}_{m\times r}=\{S_{i,j;m,s}\mid 0\le i\le m-1 \mbox{ and } 0\le j\le s-1 \}$. 

We call any nonempty ${\cal A} \subseteq {\cal S}_{m \times s}$ a \emph{rectangular} 
$m\times s$-\emph{pattern}, and  
 $(m,s)$ the \emph{widths vector} of ${\cal A}$. A rectangular $m\times s$-pattern is called a
\emph{rectangular} $m\times s$-{labyrinth pattern}, if it satisfies Properties \ref{prop1}, \ref{prop2}, 
and \ref{prop3}, in Section \ref{sec:Definition}. 
Let $\{{\cal A}_k\}_{k=1}^{\infty}$ 
be a sequence of non-empty rectangular patterns and $\{(m_k,s_k)\}_{k=1}^{\infty}$ be 
the sequence of the corresponding widths vectors, i.e., for all $k\ge 1$ we have 
${\cal A}_k\subseteq {\cal S}_{m_k\times s_k}$. 
We denote $m(n)=\prod_{k=1}^n m_k$, and $s(n)=\prod_{k=1}^n s_k$, for all $n \ge 1$. In this case 
the set ${\cal W}_n \subset {\cal S}_{m(n)\times s(n)}$ of white rectangles of level $n$, 
 is defined as in Equation \eqref{eq:W_n}, Section \ref{sec:Construction}, 
 and the set ${\cal B}_n$ of black rectangles of level $n$, correspondingly. 
 
\par The results shown in Sections $4$ and $5$ also hold (with analogous proofs) in the case 
of \emph{rectangular mixed labyrinth sets and fractals}. Here, in the proof of 
Theorem \ref{theorem:dendrite} we take into account the fact that the diameter of any rectangle 
in ${\cal W}_n$
is strictly less than $\frac{\sqrt{2}}{2^n}$. The results in Proposition \ref{prop:exits_coordinates} and 
Proposition \ref{prop:Boxcounting} hold with small modifications, which we skip here. 
\par We remark that rectangular mixed labyrinth fractals are related to the general Sierpi\'nski 
carpets studied by McMullen \cite{McMullen}. We mention that, on the one hand, McMullen uses the same pattern at each step 
of the construction, and, on the other hand, no restrictions are imposed on the pattern (except that it is not trivial).  

\section{Wild labyrinth fractals and mixed wild labyrinth fractals}\label{sec:Wild}

\begin{figure}
\begin{center}
 \begin{tikzpicture}[scale=.35]

\draw[line width=2pt] (0,0) rectangle (14,14);
\draw[line width=2pt] (2, 0) -- (2,14);
\draw[line width=2pt] (4, 0) -- (4,14);
\draw[line width=2pt] (6, 0) -- (6,14);
\draw[line width=2pt] (8, 0) -- (8,14);
\draw[line width=2pt] (10, 0) -- (10,14);
\draw[line width=2pt] (12, 0) -- (12,14);
\draw[line width=2pt] (0, 2) -- (14,2);
\draw[line width=2pt] (0, 4) -- (14,4);
\draw[line width=2pt] (0, 6) -- (14,6);
\draw[line width=2pt] (0, 8) -- (14,8);
\draw[line width=2pt] (0, 10) -- (14,10);
\draw[line width=2pt] (0, 12) -- (14,12);
\filldraw[fill=black, draw=black] (12,0) rectangle (14, 2);
\filldraw[fill=black, draw=black] (2,2) rectangle (6, 4);
\filldraw[fill=black, draw=black] (8,2) rectangle (10, 6);
\filldraw[fill=black, draw=black] (0,4) rectangle (4, 6);
\filldraw[fill=black, draw=black] (0,8) rectangle (4, 14);
\filldraw[fill=black, draw=black] (10,4) rectangle (14, 6);
\filldraw[fill=black, draw=black] (4,10) rectangle (6, 14);
\filldraw[fill=black, draw=black] (6,6) rectangle (8, 8);
\filldraw[fill=black, draw=black] (10,8) rectangle (14, 10);
\filldraw[fill=black, draw=black] (8,10) rectangle (14, 14);

\hspace{0.4cm}
\draw[line width=2pt] (15,0) rectangle (27,12);
\draw[line width=2pt] (17, 0) -- (17,12);
\draw[line width=2pt] (19, 0) -- (19,12);
\draw[line width=2pt] (21, 0) -- (21,12);
\draw[line width=2pt] (23, 0) -- (23,12);
\draw[line width=2pt] (25, 0) -- (25,12);
\draw[line width=2pt] (15, 2) -- (27,2);
\draw[line width=2pt] (15, 4) -- (27,4);
\draw[line width=2pt] (15, 6) -- (27,6);
\draw[line width=2pt] (15, 8) -- (27,8);
\draw[line width=2pt] (15, 10) -- (27,10);
 \filldraw[fill=black, draw=black] (15,0) rectangle (17, 4);
 \filldraw[fill=black, draw=black] (19,0) rectangle (27, 2);
\filldraw[fill=black, draw=black] (23,2) rectangle (27, 4);
 \filldraw[fill=black, draw=black] (19,4) rectangle (21, 6);
 \filldraw[fill=black, draw=black] (15,6) rectangle (17, 8);
 \filldraw[fill=black, draw=black] (23,6) rectangle (25, 10);
 \filldraw[fill=black, draw=black] (15,8) rectangle (19, 10);
 \filldraw[fill=black, draw=black] (21,10) rectangle (27, 12);
 \end{tikzpicture}
\end{center}
\caption{Examples: two wild labyrinth patterns, both vertically and horizontally blocked}
\label{fig:wild_pattern}
\end{figure}
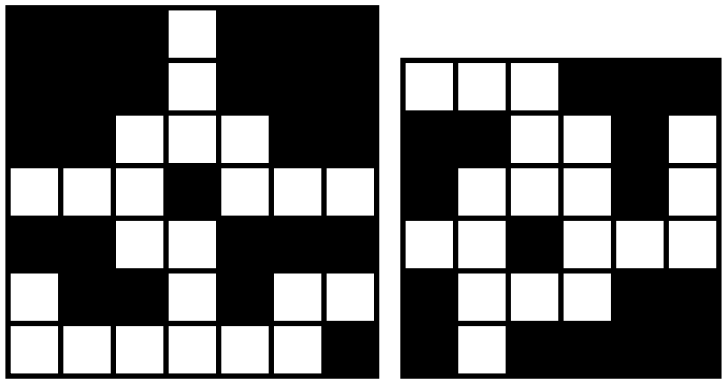
In the setting of self-similar labyrinth fractals \cite{laby_4x4, laby_oigemoan}, i.e., if 
${\cal A}_k={\cal A}_1$, 
for all $k\ge 1,$ let us weaken the conditions that define labyrinth patterns: 
instead of asking Property \ref{prop1} and Property \ref{prop2} to be satisfied, we
use here the following two properties.

\begin{property}[``Wild Property 1''] \label{prop:wild1}${\mathcal G} (A_1)$ is a connected 
graph.

\end{property}
\begin{property}[``Wild Property 2''] \label{prop:wild2} 
${\cal A}$ has at least one vertical exit exit pairpair, and at least one horizontal exit pair.
\end{property}
We call a pattern that satisfies Property \ref{prop:wild1}, Property \ref{prop:wild2}, and Property \ref{prop3} a \emph{wild 
labyrinth pattern}. We note that every labyrinth pattern is also a wild labyrinth pattern. 
Figure \ref{fig:wild_pattern} shows on the left a wild labyrinth pattern with two left and two right exits, 
and on the right a wild labyrinth patternd whose graph contains cycles.

We call the self-similar limit set generated by a wild labyrinth pattern a \emph{wild labyrinth fractal}. 
We call \emph{mixed wild labyrinth fractal} a mixed labyrinth fractal generated by a sequence of wild labyrinth patterns.

\begin{proposition}\label{prop:mixed wild connected}
 Every (mixed) wild labyrinth fractal is connected.
\end{proposition}
One way to prove the above proposition is by using recent results on connected generalised Sierpi\'nski carpets 
\cite{connected_general_carpets}.

We call a wild labyrinth pattern horizontally blocked if each of its horizontal exit pairs (i.e., a left  exit together with the corresponding right exit) is blocked, i.e. there is at least one black square in the row that contains it. 

\begin{conjecture}\label{conj:wild blocked}
 If ${\cal A}$ is a horizontally and vertically blocked wild labyrinth pattern, then 
the self-similar wild 
labyrinth fractal $L_{\infty}$ generated by ${\cal A}$ has the property that for any two distinct 
points $x,y \in L_{\infty}$ the length of an arc that connects them in $L_{\infty}$ is infinite.
\end{conjecture}
In the case of wild labyrinth fractals Lemma \ref{lemma:Construction} does not hold in general. 
Moreover, in this case the path between two exits in ${\cal G}({\cal W}_n)$ is in general not unique.
The pattern in Figure \ref{fig:counterexample_3} generates a (self-similar) 
wild labyrinth fractal with the following property: the squares of level $2$ that lie on 
the shortest path in ${\cal G}({\cal W}_2)$ from the top to the bottom exit in ${\cal G}({\cal W}_2)$ 
are not contained in the white squares which correspond to the shortest path between the top and bottom exit in
${\cal G}({\cal W}_1)$, as it can be easily checked, by comparing the paths in ${\cal G}({\cal W}_1)$ and then in ${\cal G}({\cal W}_2)$, that connect the top and bottom exit in each case. This gives an example for the case when 
the methods used in Proposition \ref{prop:recursion} do not work for the construction of the shortest 
arc between two exits as they worked for 
mixed labyrinth fractals and self-similar labyrinth fractals.

\begin{figure}
\begin{center}
 \begin{tikzpicture}[scale=.25]
\draw[line width=2pt] (0,0) rectangle (18,18);
\draw[line width=2pt] (2, 0) -- (2,18);
\draw[line width=2pt] (4, 0) -- (4,18);
\draw[line width=2pt] (6, 0) -- (6,18);
\draw[line width=2pt] (8, 0) -- (8,18);
\draw[line width=2pt] (10, 0) -- (10,18);
\draw[line width=2pt] (12, 0) -- (12,18);
\draw[line width=2pt] (14, 0) -- (14,18);
\draw[line width=2pt] (16, 0) -- (16,18);
\draw[line width=2pt] (0, 2) -- (18,2);
\draw[line width=2pt] (0, 4) -- (18,4);
\draw[line width=2pt] (0, 6) -- (18,6);
\draw[line width=2pt] (0, 8) -- (18,8);
\draw[line width=2pt] (0, 10) -- (18,10);
\draw[line width=2pt] (0, 12) -- (18,12);
\draw[line width=2pt] (0, 14) -- (18,14);
\draw[line width=2pt] (0, 16) -- (18,16);
\filldraw[fill=black, draw=black] (0,0) rectangle (8, 2);
\filldraw[fill=black, draw=black] (10,0) rectangle (18, 2);
\filldraw[fill=black, draw=black] (0,2) rectangle (6, 4);
\filldraw[fill=black, draw=black] (14,2) rectangle (18, 8);
\filldraw[fill=black, draw=black] (0,4) rectangle (4, 6);
\filldraw[fill=black, draw=black] (0,6) rectangle (2, 8);
\filldraw[fill=black, draw=black] (8,4) rectangle (12, 14);
\filldraw[fill=black, draw=black] (14,10) rectangle (18, 18);
\filldraw[fill=black, draw=black] (10,16) rectangle (14, 18);
\filldraw[fill=black, draw=black] (6,6) rectangle (8, 12);
\filldraw[fill=black, draw=black] (4,8) rectangle (6, 10);
\filldraw[fill=black, draw=black] (0,10) rectangle (2, 12);
\filldraw[fill=black, draw=black] (0,12) rectangle (4, 14);
\filldraw[fill=black, draw=black] (0,14) rectangle (6, 16);
\filldraw[fill=black, draw=black] (0,16) rectangle (8, 18);

\end{tikzpicture}
\end{center}
\caption{A wild (horizontally and vertically blocked) labyrinth pattern for which the squares in the shortest path 
from the top exit to the bottom exit in ${\cal G}({\cal W}_2)$ does not lie whithin
the shortest path from the top exit to the bottom exit in ${\cal G}({\cal W}_1)$}
\label{fig:counterexample_3}
\end{figure}
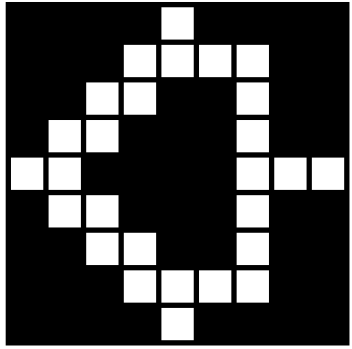
{\bf Remark.} Conjecture \ref{conj:wild blocked} can be also formulated for mixed wild labyrinth fractals, with some restrictions.
\section{Connections with results about Sierpi\'nski carpets}\label{sec:Totally disco}

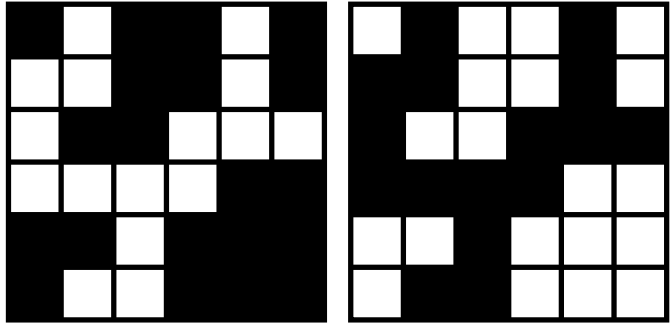
\begin{figure}
\begin{center}
 \begin{tikzpicture}[scale=.35]

\draw[line width=2pt] (0,0) rectangle (12,12);
\draw[line width=2pt] (2, 0) -- (2,12);
\draw[line width=2pt] (4, 0) -- (4,12);
\draw[line width=2pt] (6, 0) -- (6,12);
\draw[line width=2pt] (8, 0) -- (8,12);
\draw[line width=2pt] (10, 0) -- (10,12);
\draw[line width=2pt] (0, 2) -- (12,2);
\draw[line width=2pt] (0, 4) -- (12,4);
\draw[line width=2pt] (0, 6) -- (12,6);
\draw[line width=2pt] (0, 8) -- (12,8);
\draw[line width=2pt] (0, 10) -- (12,10);
\filldraw[fill=black, draw=black] (0,0) rectangle (2, 4);
\filldraw[fill=black, draw=black] (2,2) rectangle (4, 4);
\filldraw[fill=black, draw=black] (6,0) rectangle (12, 4);
\filldraw[fill=black, draw=black] (8,4) rectangle (12, 6);
\filldraw[fill=black, draw=black] (2,6) rectangle (6, 8);
\filldraw[fill=black, draw=black] (4,8) rectangle (8, 12);
\filldraw[fill=black, draw=black] (0,10) rectangle (2, 12);
\filldraw[fill=black, draw=black] (10,8) rectangle (12, 12);

\draw[line width=2pt] (13,0) rectangle (25,12);
\draw[line width=2pt] (15, 0) -- (15,12);
\draw[line width=2pt] (17, 0) -- (17,12);
\draw[line width=2pt] (19, 0) -- (19,12);
\draw[line width=2pt] (21, 0) -- (21,12);
\draw[line width=2pt] (23, 0) -- (23,12);
\draw[line width=2pt] (13, 2) -- (25,2);
\draw[line width=2pt] (13, 4) -- (25,4);
\draw[line width=2pt] (13, 6) -- (25,6);
\draw[line width=2pt] (13, 8) -- (25,8);
\draw[line width=2pt] (13, 10) -- (25,10);
 \filldraw[fill=black, draw=black] (15,0) rectangle (19, 2);
 \filldraw[fill=black, draw=black] (17,2) rectangle (19, 4);
\filldraw[fill=black, draw=black] (13,4) rectangle (21, 6);
 \filldraw[fill=black, draw=black] (13,6) rectangle (15, 8);
 \filldraw[fill=black, draw=black] (19,6) rectangle (25, 8);
 \filldraw[fill=black, draw=black] (13,8) rectangle (17, 10);
 \filldraw[fill=black, draw=black] (21,8) rectangle (23, 10);
 \filldraw[fill=black, draw=black] (15,10) rectangle (17, 12);
 \filldraw[fill=black, draw=black] (21,10) rectangle (23, 12);
 \end{tikzpicture}
\end{center}
\caption{A ($6 \times 6$-labyrinth) pattern ${\cal A}$ and its complementary pattern ${\overline{\cal A}}$}\label{fig:complement}
\end{figure}
If ${\cal A}\in {\cal S}_m$ is an $m$-pattern, we call \emph{the complementary pattern of } ${\cal A}$ the pattern denoted by 
$\overline{\cal A}$ that is defined by $\overline{\cal A}={\cal S}_m \setminus {\cal A}$, i.e., the pattern obtained by 
recolouring the squares in ${\cal A}_k$ such that
the black squares are recoloured in white squares and the white squares are recoloured in black, see, e.g., 
Figure \ref{fig:complement}.
With methods and results of a recent paper on totally disconnected Sierpi\'nski carpets 
\cite{totally_disco} one can prove the following result.
\begin{proposition}
 If $({\cal A}_k)_{k \ge 0}$ is a sequence of labyrinth patterns, then the limit set generated by the sequence of the 
corresponding complementary
patterns $(\bar{\cal A}_k)_{k \ge 0}$ is a totally disconnected generalised Sierpi\'nski carpet.
\end{proposition} 
{\bf Remark} We note that the case of generalised Sierpi\'nski carpets generated by a sequence of complementary patterns
of labyrinth patterns is an example of a situation when the main result about totally disconnected 
Sierpi\'nski carpets \cite[Theorem 1]{totally_disco}
 holds under conditions, that differ partially
from the sufficient conditions of total disconnectedness given in the cited paper.
\\[0.2cm]
{\bf Acknowledgement.} The authors thank the referee for his/her useful comments.

\end{document}